\documentclass[11pt,english,a4paper]{smfart}
\usepackage[english]{babel}
\usepackage{xcolor}
\usepackage{pbsi}
\usepackage[T1]{fontenc}
\usepackage{stmaryrd}
\usepackage{mathabx}
\usepackage[mathscr]{euscript}
\usepackage[pagebackref,colorlinks=true,linkcolor=blue,citecolor=blue,urlcolor=red, hypertexnames=true]{hyperref}
 \usepackage[abbrev,backrefs,nobysame]{amsrefs}
 \usepackage{caption}
 \usepackage{blkarray}
 \usepackage{enumerate}
\usepackage{amsmath,amssymb,bm,amsfonts}
\usepackage[all]{xy}
\entrymodifiers={+!!<0pt,\fontdimen22\textfont2>}

\usepackage{mathrsfs}

\usepackage{graphicx}

\usepackage{color}

\newtheorem{theorem}{Theorem}[section]
\newtheorem{lemma}[theorem]{Lemma}

\newtheorem{corollary}[theorem]{Corollary}
\newtheorem{proposition}[theorem]{Proposition}

\newtheorem*{coro*}{Corollary}

{\theoremstyle{definition}}

{\theoremstyle{definition}\newtheorem{example}[theorem]{Example}}

\theoremstyle{definition}
\newtheorem{definition}[theorem]{Definition}
\newtheorem{question}[theorem]{Question}
\newtheorem{fact}[theorem]{Fact}
\newtheorem{claim}[theorem]{Claim}
\newtheorem*{remark*}{Remark}

{\theoremstyle{definition}\newtheorem{remark}[theorem]{Remark}}

\def\D{\ensuremath{\mathbb D}}
\def\K{\ensuremath{\mathbb K}}
\def\T{\ensuremath{\mathbb T}}
\def\R{\ensuremath{\mathbb R}}
\def\Z{\ensuremath{\mathbb Z}}

\def\C{\ensuremath{\mathbb C}}

\def\N{\ensuremath{\mathbb N}}

\def\sp{{\rm span}}

\def\bth{\begin{theorem}}
\def\blm{\begin{lemma}}
\def\bpr{\begin{proposition}}
\def\bpf{\begin{proof}}
\def\epf{\end{proof}}
\def\epr{\end{proposition}}
\def\elm{\end{lemma}}
\def\eth{\end{theorem}}
\def\bco{\begin{corollary}}
\def\eco{\end{corollary}}
\def\be{\begin{enumerate}}
\def\ee{\end{enumerate}}
\def\bea{\begin{enumerate}[\rm (a)]}
\def\beun{\begin{enumerate}[\rm (1)]}
\def\bei{\begin{enumerate}[\rm (i)]}
\def\sgn{{\rm sgn}}
\def\llb{\llbracket}
\def\rrb{\rrbracket}
\def\bdf{\begin{definition}}
\def\edf{\end{definition}}
\newcommand{\pss}[2]{\ensuremath{{\langle #1,#2\rangle}}}

\newcommand{\ba}[1]{\overline{#1}}

\newcommand{\bbu}{\mathcal{B}_{1}}
\newcommand{\ide}[2]{\mathbf{i}_{\,#1\!,\,#2}}

\newcommand{\gd}{G_{\delta }}

\newcommand{\bbx}{{\mathcal{B}}_{1}(X)}

\newcommand{\wot}{\texttt{WOT}}

\newcommand{\sot}{\texttt{SOT}}
\newcommand{\sote}{\texttt{SOT}\mbox{$^{*}$}}
\newcommand{\sotb}{\texttt{SOT}\mbox{$_{*}$}}

\newcommand{\bx}{{\mathcal B}(X)}

\newcommand{\bu}{\mathcal{B}_{1}}

\newcommand{\lp}{\ell_{p}}

\author[S. Grivaux]{Sophie Grivaux}
\address[S. Grivaux]{CNRS, Univ. Lille, UMR 8524 - Laboratoire Paul
Painlev\'e, F-59000 Lille, France}
\email{sophie.grivaux@univ-lille.fr}
\author[\'{E}. Matheron]{\'{E}tienne Matheron}
\address[\'{E}. Matheron]{Univ. Artois, UR 2462 - Laboratoire de Math\'{e}matiques de Lens (LML)\\ F-62300 Lens, France}
\email{etienne.matheron@univ-artois.fr}
\author[Q. Menet]{Quentin Menet}
\address[Q. Menet]{Service de Probabilit\'e et Statistique, D\'epartement de Math\'ematique\\ Universit\'{e} de Mons\\ Place du Parc 20\\ 7000 Mons, Belgium}
\email{quentin.menet@umons.ac.be}

\title[Generic sets of $\ell_p\,$-$\,$contractions]{Generic properties of $\ell_p\,$-$\,$contractions and\\ similar operator topologies}
\begin{document}

\begin{abstract}
%
If $X$ is a separable reflexive Banach space, there are several natural Polish topologies on $\bbx$, the set of contraction operators on $X$ (none of which being clearly ``more natural'' than the others), and hence several \textit{a priori} different notions of genericity -- in the Baire category sense -- for properties of contraction operators.   So it makes sense to investigate to which extent the generic properties, \textit{i.e.} the comeager sets, really depend on the chosen topology on $\bbx$. In this paper, we focus on $\ell_p\,$-$\,$spaces, $1<p\neq 2<\infty$. We show that for some pairs of natural Polish topologies on $\mathcal B_1(\ell_p)$, the comeager sets are in fact the same; and our main result asserts that  for $p=3$ or $3/2$ and in the real case, \emph{all} topologies on $\mathcal B_1(\ell_p)$ lying between the Weak Operator Topology and the Strong$^*$ Operator Topology share the same comeager sets. Our study relies on the consideration of continuity points of the identity map for two different topologies on $\mathcal{B}_1 (\ell_p)$.  The other essential ingredient in the proof of our main result is a careful examination of norming vectors for finite-dimensional contractions of a special type.%
%
\end{abstract}

 \thanks{This work was supported in part by
the project FRONT of the French
National Research Agency (grant ANR-17-CE40-0021) and by the Labex CEMPI (ANR-11-LABX-0007-01). The third author is a Research Associate of the Fonds de la Recherche Scientifique - FNRS}
\thanks{The second author is indebted to Miguel Mart\'in, Javier Meri and Debmalya Sain for very pleasant and interesting discussions regarding norming vectors and the contents of the present paper.}

\maketitle

\section{Introduction and main results}\label{Intro}

\subsection{Polish spaces of operators and typical properties}\label{Subsection 1.1}
Let $X$ be a real or complex infinite-dimensional separable Banach space.  
Denote by $\mathcal B(X)$ the space of all bounded operators on $X$, and by $\bbx$ the unit ball of $\mathcal B(X)$, \mbox{\it i.e.} the set of {contraction operators} on $X$. 
In this paper, we will be interested in \emph{typical properties} of elements of \(\bu(X)\) in the sense of Baire Category. More precisely, the setup is the following: we endow the ball \(\bu(X)\) with a topology \(\tau \) which turns it into a \emph{Polish} (i.e. separable and completely metrizable) space. A property (P) of elements of \(\bu(X)\) is said to be \emph{typical for} \(\tau \),  or \(\tau \)-\emph{typical} if the set of operators satisfying (P) is comeager in \((\bu(X),\tau )\), i.e. contains a dense \(G_{\delta }\) subset of \((\bu(X),\tau )\). A typical property is thus a property that is satisfied by ``quasi-all'' contractions (in the Baire Category sense). We will mainly focus on the case where \(X\) is the real or complex \(\lp\,\)-$\,$space, \(1\le p<\infty\).
\par\smallskip
The ball $\bbx$ is usually not a Polish space when endowed with the operator norm topology (as it is not separable), and some weaker topologies are to be considered in order to turn $\bbx$ into a Polish space. We will consider in this paper four natural topologies on $\bx$: the \emph{Strong Operator Topology} (\sot), the \emph{Weak Operator Topology} (\wot), the \emph{Strong$^*$ Operator Topology} (\sote), and a topology that might be called the ``{Dual Strong Operator Topology}'', which we denote by \sotb. Recall that \sot\ is the topology of pointwise convergence, that \wot\ is the topology of weak pointwise convergence, and that \sote\ is the topology of pointwise convergence for operators and their adjoints. As to \sotb, it is the topology of pointwise convergence for the adjoints. In other words, if $(T_i)$ is a net in $\mathcal B(X)$ and if $T\in\mathcal B(X)$, then 
\begin{align*}
T_i\xrightarrow{\sot}T&\iff T_ix\xrightarrow{\Vert\,\cdot\,\Vert} Tx\quad\hbox{for all $x\in X$},\\
T_i\xrightarrow{\wot}T&\iff T_ix\xrightarrow{w} Tx\quad\hbox{for all $x\in X$},\\
T_i\xrightarrow{\sote}T&\iff T_i\xrightarrow{\sot}T\quad{\rm and}\quad T_i^*\xrightarrow{\sot}T^*,\\
T_i\xrightarrow{\sotb}T&\iff  T_i^*\xrightarrow{\sot}T^*.
\end{align*}

\noindent The space $\mathcal B(X)$ is not a Baire space for any of these topologies, but its unit ball $\bbx$ is nicer: it is well known that since $X$ is separable, $(\bbx,\sot)$ is Polish; that if $X^*$ is separable then $(\bbx,\sote)$ is Polish; and that if $X$ is reflexive then $(\bbx,\wot)$ is compact and metrizable (hence Polish) and $(\bbx,\sotb)$ is Polish. Note also that obviously $\sote=\sot\vee\sotb$, \mbox{\it i.e.} \sote\ is the coarsest topology which is finer than both \sot\ and \sotb, and that $\wot\subset\sot\cap\sotb$. The topology \sotb\ may look a bit artificial, but our results will show that it is indeed natural to consider it.
\par\smallskip
The study of typical properties of  contractions was initiated by Eisner \cite {E} and Eisner-M\'atrai  \cite{EM} in a purely Hilbertian setting, and further developed in \cite{GMM1} and \cite{GMM2}.
The setting of the monograph \cite{GMM1} is also mainly Hilbertian, and it is focused on typical properties of operators connected to linear dynamics (existence of orbits with various properties, existence of non-trivial measures...). The works \cite{GMM2} and \cite{GM} are set in the broader context of (complex) \(\lp\,\)-$\,$spaces, and centre around the following question: is it true that a typical \(T\in\bu(\lp)\) has a non-trivial invariant subspace? The underlying motivation of this question is, of course, the famous  Invariant Subspace Problem, which asks (for a given Banach space $X$) whether any bounded operator \(T\) on \(X\) admits a closed subspace \(M\) with \(M\neq\{0\}\) and 
\(M\neq X\) such that \(T(M)\subseteq M\). The answer is known to be negative in general (\cite{Enflo}, \cite{R1}), and examples of operators without non-trivial invariant subspaces have been constructed by Read on \(\ell_{1}\) or \(c_{0}\) (\cite{R2}, \cite{R3}, see also \cite{GR} for a unifying approach to such constructions). Let us also mention the existence, proved by Argyros and Haydon in \cite{AH}, of separable infinite-dimensional spaces \(X\) on which every operator does have a non-trivial invariant subspace. Despite considerable efforts, no answer (positive or negative) to the Invariant Subspace Problem have been obtained so far for any reflexive $X$, which was a motivation for studying its \textit{a priori} more tractable ``generic'' version in \cite{GMM2}.
\par\smallskip 
One would rather naturally expect typical properties of contractions \(T\in(\bu(X),\tau )\) to depend heavily on the Polish topology $\tau$ under consideration. This intuition is quite right in the Hilbertian  setting. Indeed, in the case of the complex \(\ell_{2}\,\)-$\,$space, it is proved in \cite{E} that a typical \(T\in(\bbu(\ell_2),\wot)\) is unitary, whereas in \cite{EM} the following  surprising result is obtained: a typical \(T\in(\bbu(\ell_2),\sot)\) is unitarily equivalent to $B_\infty$, the backward shift of infinite multiplicity acting on \(\ell_{2}(\Z_{+},\ell_{2})\). In particular, the typical situation for $\wot$ (resp. $\sot$) is that $T$ (resp. $T^*$) is an isometry. As to the topology \sote, it is proved in \cite{GMM1} that  a typical $T\in(\mathcal B_1(\ell_2), \sote)$ is such that  $2T$ and $2T^*$ are hypercyclic (\textit{i.e.} they admit vectors with dense orbit), so that $T$ and $T^*$ are very far from being isometries.
\par\smallskip
It is clear from these results that a typical \(T\in\bu(\ell_{2})\) for any one of the topologies \wot, \sot, \sotb\ has a non-trivial invariant subspace. It is much less clear that this remains true for the topology \sote; but it is indeed the case thanks to a  deep result of Brown, Chevreau and Pearcy \cite{BCP}, according to which any \(T\in\bu(\ell_{2})\) whose spectrum $\sigma(T)$ contains the unit circle \(\T\) has a non-trivial invariant subspace. As it can be shown (see \cite{EM}) that a typical \(T\in(\bu(\ell_{2}),\sote)\) is such that \(\sigma (T)=\ba{\,\D}\), the closed unit disk in $\C$, it follows that an \sote-$\,$typical \(T\in\bu(\ell_{2})\) does have a non-trivial invariant subspace. 
\par\smallskip
Much less is known on (complex) \(\lp\,\)-$\,$spaces: while it is still true that an \sot$\,$-$\,$typical \(T\in\bu(\ell_{1})\) has a non-trivial invariant subspace (see \cite{GMM2}), the question remains widely open for typical operators \(T\in\bu(\lp)\), \(1<p\neq 2<\infty\), whatever the topology we consider on \(\bu(\lp)\) among our favourite ones. However, at least for \(p>2\), an interesting link between typical properties of contractions on \(\lp\) for \emph{two different topologies} is uncovered in \cite{GMM2}: if $p>2$, then any comeager subset in \((\bu(\lp),\sote)\) is also comeager in \((\bu(\lp),\sot)\); in other words, a property of $\ell_p\,$-$\,$contractions which is typical for \sote\ is also typical for \sot. Since it is not too hard to show that an \sote-$\,$typical \(T\in\bu(\lp)\) is such that \(2T^{*}\) is hypercyclic (which implies in particular that $T$ has no eigenvalue), it follows that an  \sot$\,$-$\,$typical \(T\in\bu(\lp)\) has no eigenvalue \cite{GMM2}. It should be pointed out that we are unable to provide a proof of this result which does not make use of the topology \sote.

\subsection{Main results}\label{Subsection 1.2}
With the above result from \cite{GMM2} in mind, our purpose in the present paper is to undertake a systematic study of the links between comeager subsets of \(\bu(\lp)\), \(1<p\neq 2<\infty\) for different topologies, with a view towards a better understanding of typical properties of contractions on \(\lp\). The basic question we consider is the following: 
\begin{question}\label{Question B}
 Let \(X\) be a real or complex \(\lp\,\)-$\,$space, with \(1<p\neq 2<\infty\). Is it true that the four topologies \wot, \sot, \sotb, and \sote\ on \(\bu(\lp)\) have the same comeager sets?
\end{question}

\smallskip We will in fact consider a formally stronger property than just having the same comeager sets. Following \cite{Lodz}, let us say that two topologies $\tau$ and $\tau'$ on an abstract set $\mathbf M$ are \emph{similar} if they have the same dense sets, or equivalently the same sets with empty interior. It is easy to see that similar topologies have the same comeager sets, and simple examples show that the converse is not true (see Section \ref{SecTwo}). Similar topologies may be extremely different. Consider for example the classical {Sorgenfrey topology} on $\R$, which is  generated by the half-open intervals $[a,b)$, $a,b\in\R$. This topology is zero-dimensional and non-metrizable, and yet it is similar to the usual topology of $\R$.

\smallskip Our first result regarding similarity of operator topologies is the following.

\begin{theorem}\label{Theorem A}
 Let \(X\) be a real or complex \(\lp\,\)-$\,$space with \(1<p\neq 2<\infty\).
 \begin{enumerate}
  \item [\emph{(1)}] If \(p>2\), the topologies \emph{\sot}\ and \emph{\sote}\ on \(\bu(X)\) are similar, and the topologies \emph{\wot}\ and \emph{\sotb}\ are similar.
  \item [\emph{(2)}] If \(1<p<2\), the topologies \emph{\sotb}\ and \emph{\sote}\ on \(\bu(X)\) are similar, and the topologies \emph{\wot}\ and \emph{\sot}\ are similar.
 \end{enumerate}

\end{theorem}
It follows immediately that a property (P) of contractions on \(\lp\), \(p>2\) is \sot$\,$-$\,$typical (resp. \wot$\,$-$\,$typical) if and only if it is \sote-$\,$typical (resp. \sotb-$\,$typical). Analogous statements hold in the case \(1<p<2\).
\par\smallskip
However, Theorem \ref{Theorem A} falls short of proving, for instance, that an \sot$\,$-$\,$typical contraction on the complex \(\lp\,\)-$\,$space for \(1<p<2\) has no eigenvalue -- which is one of the nagging questions left open in \cite{GMM2} -- while a positive answer to the Question \ref{Question B} would yield a weight of interesting corollaries concerning typical properties of contractions on \(\lp\).

\smallskip
Our main result provides a positive answer to Question \ref{Question B} in the particular case where \(X\) is a \emph{real} \(\lp\,\)-$\,$space with \(p=3\) or \(p=3/2\):

\begin{theorem}\label{Theorem C}
 If \(X\) is a real \(\lp\,\)-$\,$space with \(p=3\) or \(p=3/2\), then all topologies on  \(\bbx\) lying between the two topologies \emph{\wot} and \emph{\sote}\  are similar.
\end{theorem}
It is tempting to conjecture that the same result holds true for any real \(\lp\,\)-$\,$space with \(1<p\neq 2<\infty\), but we have been unable to prove it for any other value of \(p\) than \(p=3\) and \(p=3/2\). We do not take the risk of making the same conjecture for complex \(\lp\,\)-$\,$spaces: the situation might well be different.

\subsection{About the proofs} 
As it turns out,  the question of the similarity of two topologies \(\tau \) and \(\tau '\) on some abstract set $\mathbf M$ is very much enlightened by the consideration of the \emph{points of continuity} of the formal identity map between the spaces \((\mathbf M,\tau )\) and \((\mathbf M,\tau ')\). This  point of view will be our starting point for the proofs of Theorem \ref{Theorem A} and Theorem \ref{Theorem C}. It should be noted that the proof of Theorem \ref{Theorem A} relies heavily on some results from \cite{GMM2}, and that Theorem \ref{Theorem A}  could have been proved using only the methods of \cite{GMM2}; but the ``points of continuity approach'' makes everything much more transparent. As to the proof of Theorem \ref{Theorem C}, it requires quite a lot of additional work besides the use of points of continuity. In particular,  a thorough study of the \emph{norming vectors} for certain classes of contractions on the real \(\ell_{3}\,\)-$\,$space will be needed. 

\subsection{Organization of the paper} 
We develop the points of continuity approach in Section \ref{SecTwo}. As a first illustration we give some natural examples of similar topologies.  We also describe the continuity points of the identity map between \((\mathbf M,\tau )\) and \((\mathbf M,\tau ')\) when \(\mathbf M=\bu(\ell_{2})\) and \(\tau \), \(\tau '\) are chosen among \wot, \sot, \sotb,  \sote\ (Proposition \ref{CPH}). Finally, we show that for certain classes of Banach spaces \(X\), all contractions $T\in\bbx$ with enough norming vectors are points of $\wot\,$-$\,\sot$ continuity (Proposition \ref{KK3}); a result that will be essential for the proof of Theorem \ref{Theorem C}. Theorem \ref{Theorem A} is proved in Section \ref{SecFour}, and Theorem \ref{Theorem C} is proved in Section \ref{p=3}.  Thanks to the points of continuity approach, things boil down to showing that the contractions with enough norming vectors are $\sot\,$-$\,$dense in $\mathcal B(\ell_3)$ (Proposition \ref{keymachin}). The remaining sections are not directly related to Theorem \ref{Theorem A} and Theorem \ref{Theorem C}, but they fit naturally into the landscape. In Section \ref{SecThree}, we show that  the points of continuity approach allows to retrieve in a very direct way the main Hilbertian typicality results from \cite{E} and \cite{EM}. In Section \ref{NAsection}, we show that, in strong contrast to what happens for the operator-norm topology, an \sote-$\,$typical \(T\in\bu(\lp)\), \(1<p<\infty\) does not attain its norm. Finally, Section \ref{FredhSection} contains additional facts concerning similar topologies, other examples of points of $\wot\,$-$\,\sot$ continuity, and a few typicality results pertaining to Fredholm theory.

\subsection{Notation}
We will denote by $\overline{\,\mathbb{D}}$ the closed unit disk in $\C$, and by $\T$ the unit circle.
The letter $X$ will always denote a real or complex infinite-dimensional separable Banach space, and the scalar field will be denoted by $\K$. The closed unit ball of $X$ will be denoted by $B_X$. The canonical basis of $\ell_p$ or $c_0$ is  $(e_j)_{j\in\Z_+}$, and $(e_j^*)$ is the associated sequence of coordinate functionals. If $N\in\Z_+$ we set $E_N:=\sp(e_0,\dots ,e_N)$, and we denote by $P_N$ the canonical projection of $X$ onto $E_N$.  The letter $\mathbf M$ will stand for an abstract set potentially equipped with several topologies. If $\tau$ and $\tau'$ are  two topologies on $\mathbf M$, we denote by $\mathcal C(\tau,\tau')$ the set of all points of continuity of the identity map $\mathbf i_{\tau,\tau'}:(\mathbf M,\tau)\to (\mathbf M,\tau')$.

\section{Similarity and points of continuity}\label{SecTwo}
\subsection{Two general facts} 
We start by presenting two very simple facts concerning similarity of topologies on an abstract set $\mathbf M$. The first one is well known; see \mbox{e.g.} \cite[Proposition 9]{Andri} or \cite[Theorem 2.2]{Lodz}.

\blm\label{similarcomeager} Let $\tau$ and $\tau'$ be two topologies on $\mathbf M$. If $\tau$ and $\tau'$ are similar, then they have the same comeager sets.
\elm
\bpf It is enough to show that if $\tau$ and $\tau'$ are similar, then they have the same nowhere dense sets. Now, given a topology on $\mathbf M$, a set $E\subseteq \mathbf M$ is nowhere dense if and only if the following holds true: for any set $A\subseteq\mathbf M$ with non-empty interior, one can find a set $B$ with non-empty interior such that $B\subseteq A$ and $B\cap E=\emptyset$. The lemma follows immediately.
\epf

\begin{remark} Topologies sharing the same comeager sets may not be similar. For example, on $\mathbf M:=\R$, let $\tau$ be the usual topology and let $\tau'$ be the topology generated by $\tau\cup\{ \R\setminus\mathbb Q\}$, so that a set $V'\subseteq\R$ is $\tau'\,$-$\,$open if and only if $V'=V\cup \bigl(W\cap(\R\setminus\mathbb Q)\bigr)$ where $V,W$ are $\tau\,$-$\,$open. By the Baire Category Theorem and since $\tau\subset \tau'$, any closed set with empty interior in $(\R,\tau)$ is also closed with empty interior in $(\R,\tau')$, and hence any $\tau\,$-$\,$meager set is $\tau'$-$\,$meager. Conversely, since every $\tau'\,$-$\,$closed set $C'\subseteq\R$ has the form $C'=C\cup D$ where $C$ is $\tau\,$-$\,$closed and $D\subseteq\mathbb Q$, any closed set with empty interior in $(\R,\tau')$ is $\tau\,$-$\,$meager and hence any $\tau'$-$\,$meager set is $\tau\,$-$\,$meager. So $\tau$ and $\tau'$ share the same comeager sets; but they are not similar since $\mathbb Q$ is not dense in $(\R,\tau')$. However, for the operator topologies we will be considering, the two properties are in fact equivalent; see Proposition \ref{similarcomeagerbis}.
\end{remark}

\par\smallskip Our second lemma shows the relevance of points of continuity of identity maps when investigating the similarity of two topologies.

\blm\label{tropsimple} Let  $\tau$ and $\tau'$ be two topologies on $\mathbf M$.
\be
\item[\rm (a)] For any set $D\subseteq \mathbf M$, we have $\overline{D}^{\,\tau}\cap \mathcal C(\tau,\tau')\subseteq \overline{D}^{\,\tau'}$.
\item[\rm (b)] If $D\subseteq \mathbf M$ is $\tau\,$-$\,$dense and $\tau'$-$\,$closed, then $D\supseteq\mathcal C(\tau,\tau')$.
\item[\rm (c)] If $\mathcal C(\tau,\tau')$ is $\tau'$-$\,$dense in $\mathbf M$ and $\mathcal C(\tau',\tau)$ is $\tau\,$-$\,$dense, then $\tau$ and $\tau'$ are similar.
\ee
\elm
\bpf  (a) If $x\in \overline{D}^{\,\tau}\cap \mathcal C(\tau,\tau')$, one can find a net $(z_i)\subseteq D$ such that $z_i\xrightarrow{\tau} x$; and then $z_i\xrightarrow{\tau'} x$ because $x\in \mathcal C(\tau,\tau')$.

(b) follows immediately from (a).

(c) By symmetry, it is enough to show that  if $\mathcal C(\tau,\tau')$ is $\tau'\,$-$\,$dense in $\mathbf M$, then any $\tau\,$-$\,$dense  subset of $\mathbf M$ is $\tau'$-$\,$dense; which is obvious: 
if $D\subseteq \mathbf M$ is $\tau\,$-$\,$dense, then $\mathcal C(\tau,\tau')\subseteq \overline{D}^{\,\tau'}$ by (a), and hence $D$ is $\tau'$-$\,$dense. 
\epf

\begin{remark} The assumption of part (c) above is satisfied in particular if $\tau\subseteq \tau'$ and $\mathcal C(\tau,\tau')$ is $\tau'$-$\,$dense in $\mathbf M$. This is the situation we will mostly consider in what follows.
\end{remark}

 \smallskip 
 \subsection{Examples of similar topologies} The next proposition is a concrete illustration of Lemma \ref{tropsimple}. Recall that a Banach space $X$ is said to have the \emph{Kadec-Klee property} if the following holds true: if $(x_n)$ is a sequence in $B_X$ such that $x_n\xrightarrow{w} x$ where $\Vert x\Vert =1$, then in fact $\Vert x_n-x\Vert\to 0$. (This is the ``sequential'' definition of Huff \cite{Hu}.) For example, $\ell_p$ has the Kadec-Klee property for every $1\leq p<\infty$. One defines in the same way the  \emph{$w^*$-$\,$Kadec-Klee property} for a dual Banach space $X=Z^*$. For example, $\ell_1=c_0^*$ has the $w^*$-$\,$Kadec-Klee property; and so does $\ell_p$ for any $1<p<\infty$ since KK and $w^*$-$\,$KK  are the same properties for reflexive spaces.
 \bpr\label{KK} Let $X$ be a Banach space with separable dual {\rm (}resp. a separable dual space{\rm )} having the Kadec-Klee property {\rm (}resp. the $w^*$-$\,$Kadec-Klee property{\rm )}
 , and let $\mathbf M:=B_X$, the closed unit ball of $X$. If $\tau$ is any topology on $B_X$ finer than the weak {\rm (}resp. the weak$^*${\rm )} topology
 , weaker than the norm topology and such that the unit sphere $S_X$ is $\tau\,$-$\,$dense in $B_X$, then the topologies $\tau$ and $w$ {\rm (}resp. $\tau$ and $w^*${\rm )} are similar.
 \epr
 \bpf If $X^*$ is separable (so that $(B_X,w)$ is metrizable) and $X$ has the Kadec-Klee property, then $S_X\subseteq \mathcal C(w,\Vert\,\cdot\,\Vert)$ and hence $S_X\subseteq \mathcal C(w,\tau)$; and similarly, if $X$ is a separable dual space with the $w^*$-$\,$Kadec-Klee property then $S_X\subseteq \mathcal C(w^*, \tau)$.  So the result follows immediately from Lemma \ref{tropsimple}.
\epf

\begin{remark*} It follows in particular from Corollary \ref{KK} that if $X$ is a separable dual space with the $w^*$-$\,$Kadec-Klee property (e.g. if $X=\ell_1$), then the weak topology and the weak$^*$ topology on $B_X$ are similar. This could of course have been proved directly.
\end{remark*}
 
 \smallskip The following special case of Proposition \ref{KK} allows to construct lots of similar Polish topologies on the unit ball of a separable dual space with the $w^*$-$\,$Kadec-Klee property.

 \bco\label{KK2} Let $X=Z^*$ be a separable dual Banach space, and let $\mathbf M:=B_{Z^*}$. Let also $\mathbf q=(q_i)_{i\in I}$ be a separating countable family of continuous and $w^*$-$\,$lower semicontinuous seminorms on $Z^*$.  Denote by $\tau$ the topology on $B_{Z^*}$ generated by $\mathbf q$. 
 \be
 \item[\rm (i)] The topology $\tau$ is Polish, finer than the weak$^*$ topology restricted to $B_{Z^*}$, and weaker than the norm topology.
 \item[\rm (ii)] If $X=Z^*$ has the $w^*$-$\,$Kadec-Klee property and $\inf_{\Vert h\Vert=1} \max_{i\in F} q_i(h)=0$ for every finite set $F\subseteq I$, then the topologies $\tau$ and $w^*$ are similar.
 \ee 
 \eco
 \bpf (i) Since the seminorms $q_i$ are continuous, $\tau$ is weaker than the norm topology. 
 
 \smallskip Let us show that $\tau$ is finer than $w^*$ restricted to $B_{Z^*}$. Since each seminorm $q_i$ is $w^*$-$\,$lower semicontinuous, we can write for each $i\in I$ the set
 $\{x\in X\textrm{ ; } q_i(x)\le 1\}$ as
 $$\bigcap_{z\in A_i} \{ z^*\in Z^* ;\; \vert \langle z^*, z\rangle\vert \leq 1\}$$ for some set $A_i\subseteq Z$. Since the family $\mathbf q$ is separating, we then have $\overline{\rm span}\,\bigcup_{i\in I} A_i=Z$; and the result follows easily since we are working on the bounded set $B_{Z^*}$. 
 
\smallskip  Since $\mathbf q$ is a countable family of seminorms, the topology $\tau$ is metrizable; namely, assuming as we may that $\mathbf q$ is uniformly bounded (on bounded sets) and that $I\subseteq \N$, it is generated by the metric $d(x,y):=\sum_{i\in I} 2^{-i} q_i(x-y)$. Since $(B_{Z^*}, w^*)$ is compact and the seminorms $q_i$ are $w^*$-$\,$lower semicontinuous, it is not hard to check that $(B_{Z^*},d)$ is complete. Hence $\tau$ is completely metrizable. Moreover, $(B_{Z^*},\tau)$ is separable since $\tau$ is weaker than the norm topology of $B_{Z^*}$. So $(B_{Z^*},\tau)$ is Polish.
 
 \smallskip  
 (ii) By Proposition \ref{KK}, in order to show that $\tau$ and $w^*$ are similar, it is enough to prove that 
 $S_{X}$ is $\tau\,$-$\,$dense in $B_{X}$. Let $x_0\in B_{X}$ be arbitrary, let $\varepsilon >0$, and let $F$ be any finite subset of $I$: we need to find a vector $x\in S_{X}$ such that $q_i(x-x_0)<\varepsilon$ for all $i\in F$. By our assumption on $\mathbf q$, one can find some $h\in X$ with $\Vert h\Vert=1$ such that $q_i(h)\leq\varepsilon/2$ for all $i\in F$. Then $q_i\bigl((x_0+th)-x_0\bigr)\leq t\varepsilon/2$ for any $t\in\R_+$ and all $i\in F$. Now, we have $\Vert x_0\Vert\leq 1$ and $\Vert x_0+2h\Vert\geq 2\Vert h\Vert-\Vert x_0\Vert\geq 1$. So we can find $t\in [0,2]$ such that $\Vert x_0+th\Vert=1$; and then the vector $x:= x_0+th$ has the required properties.
  \epf
 
 \smallskip
 \begin{example} Let $1\leq p<\infty$, and let $\mathbf M$ be the unit ball of $\ell_p(\Z_+\times\Z_+)$. Consider the ``row topology'' $\tau$ and the ``column topology'' $\tau'$ on $\mathbf M$, defined as follows: $\tau$ is generated by the seminorms $q_i(x):= \left(\sum_{j=0}^\infty \vert x_{i,j}\vert^p\right)^{1/p}$, $i\in \Z_+$, and $\tau'$ is generated by the seminorms $q'_j(x):=\left(\sum_{i=0}^\infty \vert x_{i,j}\vert^p\right)^{1/p}$, $j\in \Z_+$. Then $\tau$ and $\tau'$ are {\rm (}clearly incomparable{\rm )} similar Polish topologies on $\mathbf M$.
 \end{example}
 \bpf Each seminorm $q_i$ is $w^*$-$\,$lower semicontinuous, being the supremum of the $w^*$-$\,$continuous seminorms $q_{i,N}(x):=\left(\sum_{j=0}^N \vert x_{i,j}\vert^p\right)^{1/p}$, $N\ge 0$; and we have $\bigcap_{i\in F} \ker(q_i)\neq\{ 0\}$ for any finite set $F\subseteq\Z_+$. The same is true for the seminorms $q'_j$, so the result follows from Corollary \ref{KK2}. 
 \epf
 
 \smallskip
 \begin{example} Let $1\leq p <\infty$ and let $\mathbf M:=B_{\ell_p}$. For any bounded sequence of positive real numbers $\mathbf w=(w_n)_{n\geq 0}$ and any $p\leq q\leq \infty$, denote by $\Vert\,\cdot\,\Vert_{\mathbf w,q }$ the associated weighted $\ell_q\,$-$\,$norm on $X:=\ell_p$, \mbox{\it i.e.}
 \[ \Vert x\Vert_{\mathbf w, q}:=\Vert \mathbf w x\Vert_{q}=\left( \sum_{n=0}^\infty w_n \vert x_n\vert^q\right)^{1/q}, \quad x\in \ell_p.\]
 Then $(B_{\ell_p}, \Vert\,\cdot\,\Vert_{\mathbf w,q })$ is Polish for any choice of $(\mathbf w,q)$. Moreover, if $\mathbf w$ and $\mathbf w'$ are two weight sequences such that $\inf_{n\in\Z_+} w_n=0=\inf_{n\in\Z_+} w'_n$, then the topologies generated by $\Vert\,\cdot\,\Vert_{\mathbf w,q }$ and $\Vert\,\cdot\,\Vert_{\mathbf w',q' }$ are similar for any choice of $q,q'\geq p$. However, these topologies are incomparable as soon as $\inf_{n\in\Z_+} w_n/w'_n=0= \inf_{n\in\Z_+} w'_n/w_n$.
 \end{example}
 \bpf Since $\inf_{\Vert h\Vert=1} \Vert h\Vert_{\mathbf w, q}\leq \inf_{n\in\Z_+} \Vert e_n\Vert_{\mathbf w,q}=\inf_{n\in\Z_+} w_n$, this follows again from Corollary \ref{KK2}. 
 \epf

\subsection{Operator topologies} 
Let us now consider the case where \(\mathbf M=\bbu(X)\), with two different topologies chosen within the set \(\{\wot,\sot,\sot_{*},\sote\}\). The following lemma is  well known, but we give a proof for completeness' sake. Recall that
 a map $f:E\to E'$ between topological spaces $E$ and $E'$ is said to be \emph{Borel $1$} if $f^{-1}(V)$ is $F_\sigma$ in $E$ for any open set $V\subseteq E'$. (This notion is really interesting only if all open sets in $E$ are $F_\sigma$, e.g. if $E$ is metrizable, since otherwise continuous maps need not be Borel $1$; but it makes sense for arbitrary topological spaces.) It is well known that if $E'$ is second-countable, then the set of continuity points of any Borel $1$ map $f:E\to E'$ is $G_\delta$ and comeager in $E$ (see \mbox{e.g.} \cite[Theorem 24.14]{Ke}). 
\par\smallskip
\begin{lemma}\label{Lemma 3.1}
 Let \(X\) be a Banach space, and let \(\mathbf M:=\bbu(X)\).
 \begin{enumerate}
  \item [\emph{(1)}] if \(X\) is separable, then the identity map \(\ide{\emph{\wot}}{\emph{\sot}}\) is Borel \emph{1}.
  \item[\emph{(2)}] if \(X^{*}\) is separable, then the map \(\ide{\emph{\wot}}{\emph{\sote}}\) is Borel 1, and hence  $\mathbf i_{\emph{\wot},\emph{\sot}}$, \(\ide{\emph{\wot}}{\emph{\sot}_{*}}\) and 
  $\ide{\emph{\sot}}{\emph{\sote}}$ 
  are also Borel \emph{1}.
 \end{enumerate}
\end{lemma}

\bpf (1) We have to show that any $\sot\,$-$\,$open set  $\mathcal U\subseteq \bbx$ is $\wot\,$-$\, F_\sigma$. Since $X$ is separable, $\mathcal U$ is a countable union of finite intersections of sets of the form \[\mathcal V(A,x , \varepsilon):=\bigl\{ T\in\bbx;\; \Vert  (T-A)x\Vert<\varepsilon\bigr\} ,\] where $A\in\bbx$,  $x\in X$ and $\varepsilon >0$. Now, if $T\in\bbx$, then \[ T\in \mathcal V(A,x, \varepsilon)\iff \exists n\in\N\; \Bigl( \forall x^*\in B_{X^*}\;:\; \left\vert \pss{x^*}{(T-A)x}\right\vert\leq \varepsilon-\frac1n\Bigr).\]

\noindent This shows that each set $\mathcal V(A,x, \varepsilon)$ is $\wot\,$-$F_\sigma$; and hence that $\mathcal U$ is $\wot\,$-$F_\sigma$ as well.

\smallskip (2) Assume that $X^*$ is separable, and let us show that  \(\ide{\wot}{\sote}\) is Borel 1. Let $\mathcal U\subseteq \bbx$ be $\sote\,$-$\,$open. Then $\mathcal U$ is a countable union of finite intersections of sets either of the form $\mathcal V(A,x , \varepsilon)$, or of the form $\mathcal V^*(A,x^*,\varepsilon):= \bigl\{ T\in\bbx;\; \Vert  (T-A)^*x^*\Vert<\varepsilon\bigr\} $, where $A\in\bbx$, $x^*\in X^*$ and $\varepsilon >0$. We already know that each set $\mathcal V(A,x , \varepsilon)$ is $\wot\,$-$\,F_\sigma$. Moreover, if $T\in\bbx$ then
\[ T\in \mathcal V(A,x^*, \varepsilon)\iff \exists n\in\N\; \Bigl( \forall x\in B_{X}\;:\; \left\vert \pss{x^*}{(T-A)x}\right\vert\leq\varepsilon-\frac1n\Bigr).\] 

\noindent So each set $\mathcal V(A,x^*, \varepsilon)$ is $\wot\,$-$\, F_\sigma$, and hence $\mathcal U$ is $\wot\,$-$\, F_\sigma$.
\epf
\par\smallskip

From Lemma \ref{Lemma 3.1} and since all the relevant topologies are Polish, we immediately deduce the following corollary:

\begin{corollary}\label{CPCom} If $X^*$ is separable,  $\mathcal C(\emph{\sot,\,\sote})$ is a dense $G_\delta$ subset of $(\bbx,\emph{\sot})$. If $X$ is reflexive and separable, then $\mathcal C(\emph{\wot,\,\sot})$, $\mathcal C(\emph{\wot,\,\sotb})$ and $\mathcal C(\emph{\wot,\,\sote})$ are dense $G_\delta$ subsets of $(\bbx, \emph{\wot})$.
\end{corollary}

\subsection{The Hilbertian case} When $X$ is the Hilbert space $\ell_2$, it turns out that one can identify the points of continuity for each one of the five possible pairings $(\tau,\tau')$ of the topologies \wot, \sot, \sotb, \sote\ such that $\tau\subseteq\tau'$. This will be used in Section \ref{SecThree}.
\begin{proposition}\label{CPH} Let $\mathbf M:=\mathcal B_1(\ell_2)$.
\be
\item[\rm (1)] $\mathcal C(\emph{\wot, \sot})$ is the set of all isometries.
\item[\rm (2)] $\mathcal C(\emph{\wot, \sote })$ is the set of all unitary operators.
\item[\rm (3)] $\mathcal C(\emph{\wot, \sotb})$ is the set of all co-isometries.
\item[\rm (4)] $\mathcal C(\emph{\sot, \sote})$ is the set of all co-isometries.
\item[\rm (5)] $\mathcal C(\emph{\sotb, \sote})$ is the set of all isometries.

\ee
\end{proposition}
\bpf Since the topology \wot\ is self-adjoint, it is rather clear that an operator $T$ belongs to $\mathcal C({\wot, \sote })$ if and only if $T$ and $T^*$ both belong to $\mathcal C({\wot, \sot })$. So (2) follows from (1). Also, $T$ belongs to $\mathcal C({\wot, \sotb })$ if and only if $T^*$  belongs to $\mathcal C({\wot, \sot })$, so (3) follows from (1). Finally, since \sote\ is self-adjoint, $T$ belongs to $\mathcal C({\sotb, \sote })$ if and only if $T^*$ belongs to $\mathcal C({\sot, \sote })$, so  (5) follows from   (4).
Hence we concentrate on the proofs of assertions (1) and (4).

\smallskip Since the norm of $\ell_2$ has the {Kadec-Klee Property}, it is clear that if $(T_n)$ is a sequence of operators in $\mathcal B_1(\ell_2)$ such that $T_n\xrightarrow{\wot} J$ where $J$ is an isometry, then in fact $T_n\xrightarrow{\sot}J$. In other words, any isometry belongs to $\mathcal C({\wot, \sot})$. Since \sot$\,$-$\,$convergence implies \wot$\,$-$\,$convergence for the adjoints, this implies that any co-isometry belongs to $\mathcal C({\sot,\sote})$.

\smallskip Let us denote by $\mathcal I$ the set of all isometries, and by $\mathcal I_*$ the set of all co-isometries. By what has just been said, it remains to prove that $\mathcal I\supseteq \mathcal C({\wot, \sot})$ and that $\mathcal I_*\supseteq \mathcal C({\sot, \sote})$. Moreover since $\mathcal I$ is obviously \sot$\,$-$\,$closed in $\mathcal B_1(\ell_2)$ and since $\mathcal I_*$ is \sote -$\,$closed, Lemma \ref{tropsimple} (b) tells us that we just have to check the following facts: $\mathcal I$ is \wot$\,$-$\,$dense in $\mathcal B_1(\ell_2)$, and $\mathcal I_*$ is \sot$\,$-$\,$dense. Finally, since the topology \wot\ is self-adjoint and weaker than \sot, it is in fact enough to show that $\mathcal I_*$ is \sot$\,$-$\,$dense.

\smallskip Let $T\in\mathcal B_1(\ell_2)$ be arbitrary. For any $N\in\mathbb N$, set $A_N:= TP_N$ and 
\[ T_N:= A_N+ (I-A_NA_N^*)^{1/2} B^{N+1},\]
where $B$ is the canonical backward shift on $\ell_2$. A straightforward computation shows that $T_N^*$ is an isometry, \mbox{\it i.e.} $T_N\in\mathcal I_*$. Moreover, since $A_N\xrightarrow{\sot} T$ and $\Vert I-A_NA_N^*\Vert\leq 1$, it is also clear that $T_N\xrightarrow{\sot}T$. This concludes the proof of Proposition \ref{CPH}.
\epf

\begin{remark}\label{blabla} Let $X$ be a reflexive (separable) Banach space, and let $\mathbf M:=\bbx$. The beginning of the above proof shows the following: if the norm of $X$ has the Kadec-Klee property, then any isometry $J\in\bbx$ belongs to $\mathcal C(\wot,\sot)$. It follows that if the norm of $X^*$ has the Kadec-Klee property, then any co-isometry belongs to $\mathcal C(\wot,\sotb)$ and hence to $\mathcal C(\sot,\sote)$.
\end{remark}

\subsection{Norming vectors and points of continuity}\label{norming} In this sub-section, we elaborate on Remark \ref{blabla} above. In particular, we are going to show  that if $X$ is a reflexive (separable) Banach space with the Kadec-Klee property, then any operator $T\in\bbx$ with enough {norming vectors}  belongs to $\mathcal C(\wot,\sot)$. This will turn out to be crucial for the proof of Theorem \ref{Theorem C} (see Section \ref{p=3}).
\par\medskip
Given a (separable) Banach space $X$, let us denote by $\mathcal C(w, \Vert\,\cdot\,\Vert)$ the set of all points of continuity of the identity map 
$\mathbf i_{w,\Vert \,\cdot\,\Vert}:(B_X,w)\to (B_X,\Vert\,\cdot\,\Vert)$. Note that $\mathcal C(w, \Vert\,\cdot\,\Vert)$ is contained in the unit sphere $S_X$. Moreover, if $X$ is reflexive then $\mathcal C(w, \Vert\,\cdot\,\Vert)$ is dense in $(B_X,w)$ because $(B_X,w)$ is compact metrizable (hence Polish) and the identity map $\mathbf i_{w,\Vert \,\cdot\,\Vert}$ is Borel $1$.
\bpr\label{abstractnonsense} Let $X$ be reflexive, and let $\mathbf M:=\bbx$.  If $T\in\bbx$ is such that ${\rm span}\bigl\{ z\in S_X; \; Tz\in\mathcal C(w,\Vert\,\cdot\,\Vert)\bigr\}$ is dense in $X$, then $T\in\mathcal C(\emph{\wot},\emph{\sot})$.
\epr
\bpf Let $(T_n)$ be a sequence in $\bbx$ such that $T_n\xrightarrow{\wot}T$. We have to show that $T_n\xrightarrow{\sot} T$, \mbox{\it i.e.} that $T_nx\to Tx$ for every $x\in X$. Now, if $z\in S_X$ is such that  $Tz\in\mathcal C(w,\Vert\,\cdot\,\Vert)$, then obviously $T_n z\xrightarrow{\Vert\,\cdot\,\Vert} Tz$. Since the sequence $(T_n)$ is bounded, the result follows.
\epf

\smallskip 
Recall that if $T\in\bx$, a vector $z\in X\setminus\{ 0\}$ is said to be \emph{norming} for $T$ if $\Vert Tz\Vert=\Vert T\Vert \Vert z\Vert$. We denote by $\mathcal N(T)$ the set of all norming vectors for $T$:
\[ \mathcal N(T):=\bigl\{ z\in X\setminus\{ 0\};\; \Vert Tz\Vert=\Vert T\Vert \Vert z\Vert\bigr\}.\]

\bpr\label{KK3} Let $X$ be reflexive, and assume that $X$ has the Kadec-Klee property. If $T\in\bbx$ is such that $\Vert T\Vert=1$ and ${\rm span}\bigl( \mathcal N(T)\bigr)$ is dense in $X$, then $T\in\mathcal C(\emph{\wot, \sot})$.
\epr
\bpf This follows immediately from Proposition \ref{abstractnonsense}. Indeed, by the Kadec-Klee property, we have $\mathcal C(w,\Vert\,\cdot\,\Vert)=S_X$; and this means that the set $\bigl\{ z\in S_X; \; Tz\in\mathcal C(w,\Vert\,\cdot\,\Vert)\bigr\}$ is equal to $\mathcal N(T)\cap S_X$. 
\epf

Proposition \ref{KK3} has the following consequence, which will be our main tool for proving Theorem \ref{Theorem C} in Section \ref{p=3}.

\bco\label{KK&meager} Let $X$ be reflexive, and assume that $X$ has the Kadec-Klee property. If the set $\bigl\{ T\in\bbx;\; \overline{{\rm span}}\bigl( \mathcal N(T)\bigr)=X\bigr\}$ is $\emph{\sot}\,$-$\,$dense in $\bbx$, then the topologies \emph{\wot} and \emph{\sot} on $\bbx$ are similar.
\eco
\bpf Under the above assumption, $\mathcal C(\wot,\sot)$ is \sot$\,$-$\,$dense in $\bbx$ by Proposition \ref{KK3}; so we may apply Lemma \ref{tropsimple}.
\epf

\smallskip Finally, here is another, perhaps unexpected, consequence of Proposition \ref{KK3}. 
\bco If $T\in\mathcal B_1(\ell_2)$ is such that $\Vert T\Vert=1$ and $\overline{\rm span}\,\bigl( \mathcal N(T)\bigr)=\ell_2$, then $T$ is an isometry.
\eco
\bpf This follows from Proposition \ref{KK3} and part (1) of Proposition \ref{CPH}. 
\epf

\section{Real or complex $\ell_p\,$-$\,$spaces with $p\neq 2$}\label{SecFour}
Our aim in this section is to prove Theorem \ref{Theorem A}. The key point is the following proposition.

\begin{proposition}\label{lp} 
Let $M=\mathcal B_1(\ell_p)$. If $p>2$, then $\mathcal C(\emph{\wot}, \emph{\sotb})$ and $\mathcal C(\emph{\sot}, \emph{\sote})$ are $\emph{\sote}$-$\,$dense in $M$; and if $1<p<2$, then  $\mathcal C(\emph{\wot}, \emph{\sot})$ and 
$\mathcal C(\emph{\sotb}, \emph{\sote})$ are $\emph{\sote}$-$\,$dense in $M$. 
\end{proposition}

\smallskip The proof of Proposition \ref{lp} relies on Propositions 5.15 and 5.16 from \cite{GMM2}. For the sake of readability, we restate here what is precisely needed. Recall that if $M\in\Z_+$ then $E_M=\sp(e_0,\dots ,e_M)\subseteq \ell_p$. Identifying an operator $B\in\mathcal B(E_M)$ with $P_MBP_M\in\mathcal B(\ell_p)$, we consider $\mathcal B(E_M)$ as a subspace of $\mathcal B(\ell_p)$.

\blm\label{5.15} Let $1<p<\infty$. For any \emph{\sote}-$\,$open set $\mathcal U\neq \emptyset$ in $\mathcal B_1(\ell_p)$ and any $n_0\in\Z_+$, one can find $M\geq n_0$ and $B\in\mathcal B_1(E_M)$ such that $B\in\mathcal U$ and $B$ has the following additional properties: $\Vert B\Vert=1$, and $B$ admits a norming vector $z\in E_M$ such that $\langle e_m^*, Bz\rangle\neq 0$ for all $m\in\llb 0,M\rrb$.

\elm
\bpf This follows from \cite[Proposition 5.15]{GMM2}. 
\epf

\smallskip
\blm\label{5.16} Assume that $p>2$. Let $M\in\Z_+$, and let $B\in\mathcal B_1(E_M)$. Assume that $\Vert B\Vert=1$, and that $B$ admits a norming vector $z\in E_M$ such that $\langle e_m^*, Bz\rangle\neq 0$ for all $m\in\llb 0,M\rrb$. For any $\varepsilon >0$, there exists $\delta >0$ such that the following holds true: if $ T\in {\mathcal{B}}_{1}(\ell_p)$ satisfies $\Vert P_{M}(T-B)P_{M}\Vert<\delta$, then $\Vert P_{M}(T-B) \Vert<\varepsilon$.
\elm
\bpf This is essentially \cite[Proposition 5.16]{GMM2}. We outline here an argument, which is somewhat simpler than the one provided in  \cite{GMM2}.

\smallskip Everything relies on the following elementary fact: since $p\geq 2$, there exists a constant $\alpha >0$ such that $\vert 1+w\vert^p+\vert 1-w\vert^p \geq 2+\alpha \vert w\vert^2$ for all $w\in \C$. More precisely, by  \cite[Lemma 2.1]{Kan} one can take $\alpha:=p$, so that (by homogeneity)
\begin{equation}\label{Kanformula} \vert u+v\vert^p +\vert u-v\vert^p\geq 2\vert u\vert^p+p\vert u\vert^{p-2} \vert v\vert^2\qquad\hbox{for all $u,v\in\C$}. 
\end{equation}

\smallskip Let us fix $\varepsilon>0$. 
Let also $\delta >0$ to be specified, and let $T\in\mathcal B_{1}(\ell_p)$ be such that $\Vert P_{M}(T-B)P_{M}\Vert<\delta $. Since $B=P_MBP_M$, it is enough to show that if $\delta$ is small enough, then 
$\Vert P_MT(I-P_M)\Vert <\varepsilon$. 

\smallskip
Let 
$z\in S_{E_{M}}$ be a norming vector for $B$ such that $\langle e_m^*, Bz\rangle\neq 0$ for all $m\in\llb 0,M\rrb$. Since $\Vert Bz\Vert=1$, $z\in E_M$ and $P_MBP_M=B$, we have $\Vert P_MTz\Vert >1-\delta$ and, if $\delta$ is small enough, $\vert \langle e_m^* , Tz\rangle\vert\geq \gamma/2$ for all $m\in\llb 0,M\rrb$, where $\gamma:= \min\,\bigl\{ \vert \langle e_m^* , Bz\rangle\vert;\; m\in\llb 0,M\rrb\bigr\}>0$.

\smallskip Now, let $x\in \ell_p$ with $\Vert x\Vert=1$, and let $y:= (I-P_M)x$. By $(\ref{Kanformula})$, 
we have for every $s>0$ and all $m\in\llb 0,M\rrb$:
\begin{align*}
 ps^{2}\, \vert\pss{e_{m}^{*}}{Tz}\vert^{p-2}\,\vert\pss{e_{m}^{*}}{Ty}\vert^{2}\le
 \vert \,\pss{e_{m}^{*}}{Tz+s Ty}\,\vert ^{p}+\vert \,\pss{e_{m}^{*}}{Tz-s Ty}\,\vert ^{p}
 -2\,\vert \,\pss{e_{m}^{*}}{Tz}\,\vert ^{p}\!.
\end{align*}
Summing over $m\in\llb 0,M\rrb$ and assuming that $\delta$ is small enough, it follows that
\begin{align*}
 ps ^{2}\,(\gamma/2)^{p-2}\sum_{m=0}^M\vert\,\pss{e_{j}^{*}}{Ty}\,\vert ^{2}&\le\Vert P_MT(z+s y)\Vert^{p}+\Vert P_MT(z-s y)\Vert^{p}-2\,\Vert P_MTz\Vert^{p}\\
&\le\Vert z+sy\Vert^{p}+\Vert z-s y\Vert^{p}-2(1-\delta)^{p}.
\end{align*}
Therefore, since $\Vert z\pm s y\Vert^p=\Vert z\Vert^p+s^p \Vert y\Vert^p\leq 1+s^p$ and $\sum_{m=0}^M\vert\,\pss{e_{m}^{*}}{Ty}\,\vert ^{2}\geq \Vert P_M Ty\Vert^2$ (because $p\geq 2$), we obtain, setting $c:= p(\gamma/2)^{p-2}$ and $\eta:= 1-(1-\delta)^p$, that
\[ \Vert P_MT(I-P_M)x\Vert^2\leq \frac2{c}\,\left( \frac{\eta}{s^2} +s^{p-2}\right),\]
for any $x\in S_{\ell_p}$ and all $s>0$. 

Optimising with respect to $s$, \mbox{\it i.e.} taking $s:=\left(\frac{2\eta}{p-2}\right)^{1/p}$, it follows that 
\[ \Vert P_MT(I-P_M)\Vert^2 \leq K\, \eta^{1-\frac2p},\]
for some constant $K$ independent of $T$ and $\eta$. Since $\eta =1-(1-\delta)^p$, this shows that $\Vert P_MT(I-P_M)\Vert<\varepsilon$ if $\delta$ is small enough.
\epf

\begin{proof}[Proof of Proposition \ref{lp}] By duality, it is enough to consider the case $p>2$. Moreover, since \sot\ is stronger than \wot\ and since $\sote=\sot\vee\sotb$, it is clear that $\mathcal C(\wot,\sotb)\subseteq \mathcal C(\sot,\sote)$. So it is enough to show that  $\mathcal C(\wot,\sotb)$ is \sote -$\,$dense in $\mathbf M=\mathcal B_1(\ell_p)$, $p>2$.

\smallskip For any $A,B\in \mathcal B_1(\ell_p)$, let us set
\[ d_*(A,B):=\sum_{n=0}^\infty 2^{-n} \Vert P_n(B-A)\Vert .\]

Since $\Vert P_n(B-A)\Vert =\Vert (B^*-A^*)P_n^*\Vert$ for all $n\geq 0$, the metric $d_*$ generates the topology $\sotb$. Therefore, if we define the sets
\[ \mathcal V_\eta=\Bigl\{ B\in\mathcal B_1(\ell_p);\; \exists \mathcal W\;\hbox{\wot$\,$-$\,$neighbourhood of $B$}\;:\;  d_*\hbox{-}\,{\rm diam}\bigl( \mathcal W\bigr)<\eta\Bigr\},\]
then 
\[ \mathcal C(\wot,\sotb)=\bigcap_{\eta >0} \mathcal V_{\eta}=\bigcap_{k\in\N} \mathcal V_{1/k}.\]

Now, the sets $\mathcal V_\eta$ are \wot$\,$-$\,$open in $\mathcal B_1(\ell_p)$, hence \sote -$\,$open. So, by the Baire Category Theorem applied to $(\mathcal B_1(\ell_p),\sote)$, it is enough to show that $\mathcal V_\eta$ is \sote -$\,$dense in $\mathcal B_1(\ell_p)$ for every $\eta >0$.

Let us fix $\eta >0$, and let $\mathcal U$ be a non-empty open set in $(\mathcal B_1(\ell_p),\sote)$. Choose $n_0\in\N$ such that $\sum_{n> n_0} 2^{-n}<\eta/4$. By Lemmas \ref{5.15} and \ref{5.16}, one can find  an operator $B\in  \mathcal U$ and an integer $M\geq n_0$ such that the following holds true: for any $\varepsilon >0$, there exists $\delta >0$ such that for any $T\in\mathcal B_1(\ell_p)$:
\[ \Vert P_M(T-B)P_M\Vert <\delta\implies \Vert P_M (T-B)\Vert <\varepsilon. \]

Now, choose $\varepsilon:=\eta/8$, let $\delta$ satisfy the above property, and define
\[ \mathcal W:=\bigl\{ T\in\mathcal B_1(\ell_p);\; \Vert P_M(T-B)P_M\Vert <\delta.\]

The set $\mathcal W$ is a \wot$\,$-$\,$neighbourhood of $B$; and by the choice of $\delta$, we have
\begin{align*} d_*\,\hbox{-}\,{\rm diam}(\mathcal W)&\leq 2\sum_{n=0}^M 2^{-n}\,\eta/8 +2\sum_{n>M} 2^{-n}<\eta.
\end{align*}
This concludes the proof.\epf

\par\smallskip
We can now prove Theorem \ref{Theorem A}, which we restate here 
as Theorem \ref{samemeager}.
\par\smallskip
\bth\label{samemeager} On $\mathbf M=\mathcal B_1(\ell_p)$, $p>2$, the topologies \emph{\sot}  and \emph{\sote} are similar and the topologies \emph{\wot} and \emph{\sotb}\ are similar. On 
$\mathbf M=\mathcal B_1(\ell_p)$, $1<p<2$, the topologies \emph{\sotb}  and \emph{\sote} are similar and the topologies \emph{\wot} and \emph{\sot}\ are similar. 
\eth
\bpf This follows immediately from Proposition \ref{lp} and Lemma \ref{tropsimple}.
\epf
\par\smallskip
As a direct corollary, we retrieve the result from \cite[Theorem 2.3]{GMM2} that in the complex case and for \(p>2\), a typical \(T\in(\bbu(\ell_{p}),\sot)\) is such that the operators \(2T\) and \(2T^{*}\) are hypercyclic, and obtain an analogous statement for \(1<p<2\):
\par\smallskip
\bco Let $\K=\C$. If $p>2$, then an \emph{\sot}$\,$-$\,$typical $T\in\mathcal B_1(\ell_p)$ is such that $2T$ and $2T^*$ are hypercyclic, so that in particular $T-\lambda I$ is one-to-one with dense range for every $\lambda\in\C$. If $1<p<2$, then an \emph{\sotb}-$\,$typical $T\in\mathcal B_1(\ell_p)$ is such that $2T$ and $2T^*$ are hypercyclic.
\eco
\bpf By \cite[Proposition 2.3]{GMM1}, we know that for any $1<p<\infty$, an \sote -$\,$typical $T\in\mathcal B_1(\ell_p)$ is such that $2T$ and $2T^*$ are hypercyclic. So this is also true for an $\sot\,$-$\,$typical $T$ if $p>2$ and for an $\sotb\,$-$\,$typical $T$ if $1<p<2$.
\epf
\par\smallskip
We also derive the following result for \(\wot\,\)-$\,$typical contractions on \(\ell_{p}\):
\par\smallskip

\bco Let $\K=\C$. If $1<p<2$, then a typical $T\in(\mathcal B_1(\ell_p),\emph{\wot})$ is such that $T-\lambda I$ has dense range for every $\lambda \in\C$. If $p>2$, then a typical $T\in(\mathcal B_1(\ell_p), \emph{\wot})$ has no eigenvalue.
\eco
\bpf By \cite[Proposition 2.3]{GMM1} again, an \sot$\,$-$\,$typical $T\in\mathcal B_1(\ell_p)$ is such that $2T$ is hypercyclic, and hence such that $T-\lambda I$ has dense range for every $\lambda\in\C$. So this is also true for a \wot$\,$-$\,$typical $T\in\mathcal B_1(\ell_p)$ if $1<p<2$. The second statement follows by duality.
\epf
\par\smallskip
Finally, one can  deduce from Theorem \ref{samemeager} combined with Proposition \ref{sigmae} below that a \wot$\,$-$\,$typical contraction on $\ell_p$, $p\neq 2$ has maximal essential spectrum (which  is to be compared with the forthcoming Proposition \ref{Eisner}):
\bco Let $\K=\C$ and $X=\ell_p$, $1<p\neq 2<\infty$. The essential spectrum of a typical $T\in (\bbx, \emph{\wot})$ is equal to $\overline{\,\D}$.
\eco
\bpf Since the property ``to have an essential spectrum equal to $\overline{\,\D}$'' is self-adjoint, this follows from Proposition \ref{sigmae} and Theorem \ref{samemeager}.
\epf

\section{Real $\ell_p$ with $p=3$ or $3/2$}\label{p=3}

In this section, we consider \emph{real} $\ell_p\,$-$\,$spaces, which we still denote by $\ell_p$. Our aim is to prove Theorem \ref{Theorem C}, which we restate as follows.
\bth\label{toutcapourca} If $p=3$ or $3/2$, then the topologies  \emph{$\wot$}  and \emph{$\sote$} on $\mathcal B_1(\ell_p)$ are similar.
\eth

\smallskip  Note that Theorem \ref{toutcapourca} is indeed equivalent to Theorem \ref{Theorem C}: if $\tau$ is any topology lying between {$\wot$}  and $\sote$, and if we know that {$\wot$}  and {$\sote$} have the same dense sets (which is the content of Theorem \ref{toutcapourca}), then {$\sote$} and $\tau$ clearly have the same dense sets. 

\smallskip 
Observe next that since $3$ and $3/2$ are conjugate exponents, the case $p=3/2$ follows from the case $p=3$ by duality. Also, by part (1a) of Theorem \ref{Theorem A}, we already know that for any $p>2$, $\sote$ is similar to $\sot$  on $\bbu(\ell_p)$. Hence, to prove Theorem \ref{toutcapourca}, it is enough to show that if $p=3$, then $\wot$ and $\sot$ are similar. Moreover, by Corollary \ref{KK&meager}, it is in fact enough to prove the following proposition. 
\bpr\label{keymachin} If $p=3$, then the set of all $T\in\mathcal B_1(\ell_p)$ such that $\overline\sp\,\bigl (\mathcal N(T)\bigr)=\ell_p$ is \emph{$\sot\,$-$\,$}dense in $\mathcal B_1(\ell_p)$.
\epr

\smallskip
Since the proof of Proposition \ref{keymachin} is rather convoluted, we divide this section into small sub-sections. We have tried to indicate very precisely where the assumptions that $p=3$ and $\K=\R$ seem to be really crucial. However, we would be extremely surprised if $p=3$ were the only value of $p$ for which Proposition \ref{keymachin} holds true. We would rather be inclined to believe that in fact, Proposition \ref{keymachin} holds true for all $p>2$, so that (by duality) the answer to Question \ref{Question B} is ``Yes'', in the real case at least. 

\subsection{Notation and terminology} Recall that $(e_n)_{n\in\Z_+}$ is the canonical basis of $\ell_p$, and that $(e_n^*)$ is the associated sequence of coordinate functionals. If $N\in\Z_+$, we set $E_N:=\sp(e_0,\dots ,e_N)$, and we denote by $P_N$ the canonical projection of $\ell_p$ onto $E_N$. If $N,M\in\Z_+$, we denote by $\mathcal B(E_N,E_M)$ the space of all (bounded) operators $S:E_N\to E_M$. Whenever this seems convenient, we consider $\mathcal B(E_N,E_M)$ as a subspace of $\mathcal B(\ell_p)$ or of $\mathcal B(E_N,\ell_p)$ \textit{via} the obvious identifications. Finally, if $x\in\ell_p$, the \emph{support} of $x$, denoted by ${\rm supp}(x)$, is the set $\{ n\in\Z_+;\; \langle e_n^*,x\rangle\neq 0\}$.

\subsection{Two useful facts} The following lemma, which is part of \cite[Theorem 2.2]{Kan}, will be used a number of times in the proof of Proposition \ref{keymachin}. We give a proof for completeness' sake (which is not the same as the one in \cite{Kan}).

\blm\label{Kansupp} Let $p>2$, let $E$ be a subspace of $\ell_p$, and let $S\in\mathcal B(E,\ell_p)$. Let also $x\in E$ be a norming vector for $S$. If $y\in E$ is such that ${\rm supp}(y)\cap {\rm supp}(x)=\emptyset$, then ${\rm supp}(Sy)\cap{\rm supp}(Sx)=\emptyset$.
\elm
\bpf We prove in fact the following more general result. \emph{Let $E$ be a real or complex Banach space, let $(\Omega,\mu)$ be a measure space, and let $S:E\to L_p(\Omega,\mu)$ be a bounded operator. Let also $x\in\mathcal N(S)$. If $y\in E$ is such that $(x,y)$ is a
sub$\,$-$\,\ell_q\,$-$\,$sequence for some $q>2$, \mbox{\emph{i.e}} $\Vert ax+by\Vert^q\leq \Vert ax\Vert^q+\Vert by\Vert^q$ for all $a,b\in\K$, then $(Sx)(Sy)=0$.} 

\smallskip We may assume that  $\Vert S\Vert=1$ and $\Vert x\Vert=1=\Vert y\Vert$. Then
\[ \Vert Sx+ t Sy\Vert_p^p\leq \Vert x+ty\Vert^p\leq \bigr(1+ \vert t\vert^q\bigr)^{p/q} =: \alpha(t)\qquad\hbox{for all $t\in \R$}.\]

Note that since $q>2$, the function $\alpha$ is $\mathcal C^2$-$\,$smooth on $\R$, with $\alpha(0)=1$ and $\alpha'(0)=\alpha''(0)=0$.

\smallskip
Let $\phi:\Omega\to\K$ be a measurable function such that $\vert \phi \vert\equiv 1$ and $Sx=\vert Sx\vert\,\phi$. Then
\[ \left\vert Sx+t\, Sy\right\vert= \Bigl\vert \vert Sx\vert+t\, \bar\phi Sy\Bigr\vert\geq \bigl\vert \vert Sx\vert +t\, {{\rm Re}(\bar\phi Sy)}\bigr\vert.\]
So, if we set $u:={{\rm Re}(\bar\phi Sy)}$ then 
\[ \bigl\Vert \vert Sx\vert +t\, u\bigr\Vert^p_p\leq \alpha(t)\qquad\hbox{for all $t\in \R$}.\] 

Now, consider the function $f:\R\to\R$ defined by 
\begin{align*} f(t):= \bigl\Vert \vert Sx\vert +t\, u\bigr\Vert^p_p-\alpha(t) =\int_\Omega \bigl\vert \vert Sx\vert +tu\bigr\vert^p d\mu-\alpha(t).
\end{align*}

Since $p>2$, the function $f$ is $\mathcal C^2$-$\,$smooth on $\R$, with 
\[ f'(t)= p\int_\Omega \bigl\vert \vert Sx\vert +tu\bigr\vert^{p-1} u d\mu-\alpha'(t)\quad{\rm and}\quad f''(t)=p(p-1)\int_\Omega \bigl\vert \vert Sx\vert +tu\bigr\vert^{p-2} u^2 d\mu-\alpha''(t).\]

Moreover, we have $f(t)\leq 0$ on $\R$ and $f(0)=\Vert Sx\Vert_p^p-1=0$ since $x\in\mathcal N(S)$; so $f$ has a maximum at $t=0$. Hence $f''(0)\leq 0$, \mbox{\it i.e.}
\[ \int_\Omega  \vert Sx\vert^{p-2} u^2 d\mu\leq 0.\]

Since $\vert Sx\vert^{p-2} u^2 \geq 0$, this means that $(Sx) u=0$. So we have shown that $(Sx)\,{\rm Re}(\bar\phi\, Sy)=0$. 
In the case where $\K=\R$, this yields that $(Sx)(Sy)\bar\phi=0$, and hence $(Sx)(Sy)=0$ since $\vert \phi\vert\equiv 1$. In the case where $\K=\C$,
starting from the inequality $\Vert Sx-i t Sy\Vert_p^p\leq \alpha(t)$ we obtain that $(Sx)\,{\rm Im}(\bar\phi\, Sy)=0$. Again, $(Sx)(Sy)\bar\phi=0$, and hence $(Sx)(Sy)=0$.
\epf

\smallskip The following  inequality (see \cite[Lemma 2.1]{Kan}) will also be needed. Note that this inequality has already been used in the proof of Proposition \ref{lp}.
\blm\label{KanIneq} Let $p>2$. For any $u,v\in \C$,
\[ \vert u+v\vert^p+\vert u-v\vert^p\geq 2\vert u\vert^p +p \,\vert u\vert^{p-2} \vert v \vert^2.\]
\elm

\subsection{The interval property} The next definition will be essential for us.

\begin{definition} Let $M,N\in\Z_+$, and let $S\in \mathcal B(E_N,E_M)$. We say that $S$ has the \emph{interval property} if the support of any norming vector $x$ for $S$ is an interval of the form $\llbracket 0,a\rrbracket$ for some $0\leq a\leq N$.
\end{definition}

\smallskip The importance of this definition lies in the following lemma. We can prove it for $p=3$ only, but it is quite plausible that it holds true for every $p>2$.

\blm \label{finitesupport} Assume that $p=3$. If $S\in \mathcal B(E_N,E_M)$ has the interval property, then $\mathcal N(S)\cap S_{E_N}$ is finite.
\elm

\smallskip To prove Lemma \ref{finitesupport}, we need to introduce another definition (where the assumption that $\K=\R$ is essential).

\begin{definition} Let $M,N\in\Z_+$, and let $S\in \mathcal B(E_N,E_M)$. Let also $x,y\in E_N$. We say that  \emph{$x$ and $y$ are sign-compatible relative to 
$S$}, or simply that \emph{$x$ and $y$ are sign-compatible}, if $\langle e_n^*, x\rangle\,\langle e_n^*,y\rangle\geq 0$ for all $n\in \llbracket 0,N\rrbracket$ and $\langle e_m^*, Sx\rangle\,\langle e_m^*,Sy\rangle\geq 0$ for all 
$m\in \llbracket 0,M\rrbracket$.
\end{definition}

In other words, the $n^{th}$ coordinates of $x$ and $y$ have the same sign for every $n\in \llbracket 0,N\rrbracket$, and the $m^{th}$ coordinates of $Sx$ and $Sy$ have the same sign for every $m\in \llbracket 0,M\rrbracket$ -- with the convention that $0$ has both signs $+$ and $-$.
Note that in spite of what the terminology might suggest, we have to be a little bit careful: the relation of sign-compatibility is obviously reflexive and symmetric, but it is not transitive.

\smallskip In the next lemma, the assumption that $p=3$ seems to be crucial. At least, it \emph{is} quite crucial for the proof we give.
\blm\label{sign} Assume that $p=3$. Let $M,N\in\Z_+$,  and let $S\in \mathcal B(E_N,E_M)$. Let also $x,y\in E_N$, and assume that $x$ and $y$ are sign-compatible with respect to $S$. Then the following facts hold true.
\be
\item[\rm (i)] There exists $c>0$ such that $x$ and $x+ty$ are sign-compatible for every real number $t\geq -c$.
\item[\rm(ii)] If $x, y\in\mathcal N(S)$, then $x+ty\in \mathcal N(S)$ for all $t\geq 0$. 
\item[\rm (iii)] Assume that $x,y\in\mathcal N(S)$ and that ${\rm supp}(y)\subseteq {\rm supp}(x)$. Then ${\rm supp}(Sy)\subseteq {\rm supp}(Sx)$ and, if $c$ is as in {\rm (i)}, $x+ty\in \mathcal N(S)\cup\{ 0\}$ for every $t\geq -c$.
\ee 
\elm
\bpf (i) It is clear that $x$ and $x+ty$ are sign-compatible for every $t\geq 0$; so we only need to show that there exists $c>0$ such that $x$ and $x+ty$ are sign-compatible for every $t\in [-c,0]$. But this is easy: just choose $c$ small enough to ensure that $\langle e_n^*, x\rangle\, \langle e_n^*, x+ty\rangle>0$ on $[-c,0]$ for all $n\in\llb 0,N\rrb$ such that $\langle e_n^*, x\rangle\neq 0$, and $\langle e_m^*, Sx\rangle\, \langle e_m^*, S(x+ty)\rangle>0$ on $[-c,0]$ for all $m\in\llb 0,M\rrb$ such that $\langle e_m^*, Sx\rangle\neq 0$.

\smallskip
(ii) Assume that $x,y\in\mathcal N(S)$ and, without loss of generality, that $\Vert S\Vert=1$. 

Let us consider the function $\varphi_{x,y}:\R\to \R$ defined by 
\[ \varphi_{x,y}(t):= \Vert x+ty\Vert^3-\Vert S(x+ty)\Vert^3.\]

Note that $\varphi_{x,y}(t)\geq 0$ on $\R$ because $\Vert S\Vert=1$, and $\varphi_{x,y}(0)=0$ because $x\in\mathcal N(S)$. Moreover, the function $\varphi_{x,y}$ is $\mathcal C^1$-$\,$smooth on $\R$. We have to show that $\varphi_{x,y}(t)=0$ for all $t\geq 0$.

\smallskip  Let $c$ be as in (i).  For any $n\in\llb0,N\rrb$ such that $\langle e_n^*,x\rangle\neq 0$ and $t\geq -c$, we may write $\vert \langle e_n^*,x+ty\rangle\vert =\omega_n \langle e_n^*,x+ty\rangle$, where $\omega_n=\pm 1$ does not depend on $t$; and it follows that $\vert \langle e_n^*,x+ty\rangle\vert^3$ is a polynomial function of $t\in [-c,\infty)$. Similarly, if $m\in\llb0,M\rrb$ is such that $\langle e_m^*, Sx\rangle\neq 0$, then $\vert \langle e_m^*,S(x+ty)\rangle\vert^3$ is a polynomial function of $t\in [-c,\infty)$. 
Writing down the definition of 
$\varphi_{x,y}(t)$, it follows that for $t\geq -c$,
\[ \varphi_{x,y} (t)= P(t)+ d \,\vert t\vert^3,\]
where $P$ is a polynomial of degree at most $3$ and \[ d:=\sum_{\langle e_n^*,x\rangle=0} \vert \langle e_n^*,y\rangle\vert^3- \sum_{\langle e_m^*,Sx\rangle=0} \vert \langle e_m^*,Sy\rangle\vert^3.\]
Let us write $P$ as $P(t)=c_0+c_1t+c_2 t^2 +c_3 t^3$, so that
\[ \varphi_{x,y}(t)=c_0+c_1t+c_2 t^2 +c_3 t^3+ d \,\vert t\vert^3\qquad\hbox{for all $t\geq -c$}.\]
We have $c_0=\varphi_{x,y}(0)=0$, and $c_1=\varphi_{x,y}'(0)=0$ because $\varphi_{x,y}$ attains its minimum at $t=0$; so 
\[ \varphi_{x,y}(t)=c_2 t^2 +c_3 t^3 + d\, \vert t\vert^3.\]

Next, we use a symmetry argument. Let $\varphi_{y,x}$ be the function defined by
\[ \varphi_{y,x}(t):=\Vert y+tx\Vert^3-\Vert S(y+tx)\Vert^3.\]
If  $t>0$ then $\varphi_{x,y}(t)=t^3 \varphi_{y,x}(1/t)$; hence, 
\[ \forall u>0\;:\; \varphi_{y,x} (u)=u^3 \varphi_{x,y} (1/u)= d+c_3+c_2 u.\]
Since $\varphi_{y,x}(0)=0=\varphi'_{y,x}(0)$ because $y\in\mathcal N(S)$, it follows that $d+c_3=0=c_2$. Thus, we obtain 
\[ \varphi_{x,y}(t)= d\, (\vert t\vert^3-t^3)\qquad\hbox{for all $t\geq -c$}.\]
In particular, $\varphi_{x,y}(t)=0$ for all $t\geq 0$, which gives (ii). 

\smallskip
(iii) With the notation of (ii), we have $\varphi_{x,y}(t) = 2 d\, \vert t\vert^3$ for all $t\in [-c,0]$. Since $\varphi_{x,y}\geq 0$, it follows that 
\[ d=\sum_{\langle e_n^*,x\rangle=0} \vert \langle e_n^*,y\rangle\vert^3- \sum_{\langle e_m^*,Sx\rangle=0} \vert \langle e_m^*,Sy\rangle\vert^3\geq 0.\]
Since ${\rm supp}(y)\subseteq {\rm supp}(x)$, the first sum in the right-hand side of the expression above is equal to $0$; so the second sum must also be equal to $0$, \mbox{\it i.e.} ${\rm supp}(Sy)\subseteq {\rm supp}(Sx)$. Altogether $d=0$; so $\varphi_{x,y}(t)=0$ for all $t\geq -c$, which means that $x+ty\in\mathcal N(S)\cup\{ 0\}$.

\epf

\smallskip We can now prove Lemma \ref{finitesupport}.

\begin{proof}[Proof of Lemma \ref{finitesupport}] Let us fix $S\in \mathcal B(E_N,E_M)$ with the interval property. We start with the following fact.
\begin{fact}\label{faux?} If $x,y\in\mathcal N(S)$ are sign-compatible, then $x$ and $y$ are proportional.
\end{fact} 
\begin{proof}[Proof of Fact \ref{faux?}] Towards a contradiction, assume that there exist two sign-compatible and linearly independent vectors $x, y$ in $\mathcal N(S)$. By the interval property, we may also assume that ${\rm supp}(y)\subseteq {\rm supp}(x)$.

\begin{claim}\label{c} There exists a largest $c>0$ such that $x$ and $x+ty$ are sign-compatible for all $t\geq -c$.
\end{claim}
\begin{proof}[Proof of Claim \ref{c}] By Lemma \ref{sign}, there exists at least one $c>0$ such that $x$ and $x+ty$ are sign compatible for all $t\geq -c$, and for any such $c$ we have $x+ty\in\mathcal N(S)$ for all $t\geq -c$ (we cannot have $x+ty=0$ because   $x$ and $y$ are linearly independent). Moreover, since $\langle e_0^*, u\rangle\neq 0$ for any $u\in\mathcal N(S)$ by the interval property, it cannot be that $x+t y\in \mathcal N(S)$ for all $t\in\R$. This proves our claim. 
\epf

\begin{claim}\label{??} Let $c$ be as in Claim \ref{c} and let $z:=x-c y$. Then $z\in\mathcal N(S)$ and ${\rm supp}(z)$ is strictly contained in ${\rm supp}(y)$.
\end{claim}
\begin{proof}[Proof of Claim \ref{??}]  That $z\in\mathcal N(S)$ follows from Lemma \ref{sign}. Moreover, since we have that ${\rm supp}(z)\subseteq {\rm supp}(x)$ and $x,z$ are sign-compatible, it also follows from Lemma \ref{sign} that ${\rm supp}(Sz)\subseteq {\rm supp}(Sx)$.

We first note that $z$ and $y$ are sign-compatible: indeed, since $x,y$ and $x,z$ are sign-compatible and ${\rm supp}(z)\subseteq {\rm supp}(x)$, we have $\langle e_n^*, z\rangle\, \langle e_n^*,x\rangle>0$ and $\langle e_n^*,x\rangle\, \langle e_n^*,y\rangle\geq 0$ for all $n\in{\rm supp}(z)$, so that $\langle e_n^*, z\rangle\, \langle e_n^*,y\rangle\geq 0$; and similarly, $\langle e_m^*, Sz\rangle\, \langle e_m^*,Sy\rangle\geq 0$ for all $m\in{\rm supp}(Sz)$. 

Now, assume that ${\rm supp}(z)$ is not strictly contained in ${\rm supp}(y)$. Then  ${\rm supp}(y)\subseteq {\rm supp}(z)$ by the interval property. By Lemma \ref{sign} and since $z,y$ are sign-compatible, it follows that one can find $\varepsilon >0$ such that $z$ and $z+sy$ are sign-compatible for all $s\geq -\varepsilon$, and it also follows that ${\rm supp}(Sy)\subseteq {\rm supp}(Sz)$.
 Let us show that $x$ and $x+ty$ are sign-compatible for all $t\geq -c-\varepsilon$, which will contradict the maximality of $c$. Write $t=-c+s$ with $s\geq -\varepsilon$. Then $x+ty=z+sy$, so that $z, x+ty$ are sign-compatible and  ${\rm supp}(x+ty)\subseteq {\rm supp}(z)$. So  we have $\langle e_n^*,z\rangle\, \langle e_n^*,x+ty\rangle >0$ for all $n\in {\rm supp}(x+ty)$, and also $\langle e_n^*,z\rangle \,\langle e_n^*, x\rangle \geq 0$  since $z,x$ are sign-compatible. Hence, we see that $\langle e_n^*,x\rangle\, \langle e_n^*,x+ty\rangle \geq 0$ for all $n\in {\rm supp}(x+ty)$. Similarly, using the sign-compatibility of $z, x+ty$ and the fact that ${\rm supp}(S(x+ty))\subseteq {\rm supp}(Sz)$, we get $\langle e_m^*,Sx\rangle\, \langle e_m^*,S(x+ty)\rangle \geq 0$ for all $m\in {\rm supp}(S(x+ty))$.
\epf

\begin{claim}\label{equalsupport} We have ${\rm supp}(x)={\rm supp}(y)$.
\end{claim}
\begin{proof}[Proof of Claim \ref{equalsupport}] Towards a contradiction, assume that ${\rm supp}(x)=\llbracket 0, L\rrbracket$ and that ${\rm supp}(y)=\llbracket 0, L'\rrbracket$ with $L'<L$ (recall that $S$ has the interval property and that ${\rm supp}(y)\subseteq {\rm supp}(x)$). By Claim \ref{??}, one can find $n\in \llbracket 0, N\rrbracket$ such that $\langle e_n^*,x\rangle\, \langle e_n^*,y\rangle\neq 0$ and $\langle e_n^*,z\rangle=0$. Then $n\leq L'<L$, and $n\not\in{\rm supp}(z)$. However, since $L\in {\rm supp}(x)\setminus{\rm supp}(y)$ and $z=x-cy$, we have $L\in{\rm supp}(z)$. Since $z\in\mathcal N(S)$, this contradicts the interval property.
\epf

So far, we have proved the following: whenever $x,y\in\mathcal N(S)$ are linearly independent and sign-compatible, it holds that ${\rm supp}(x)={\rm supp}(y)$. Now, if we start with two linearly independent sign-compatible vectors $x,y\in\mathcal N(S)$ with ${\rm supp}(y)\subseteq {\rm supp}(x)$, and if we apply this property to $x$ and to the vector $z=x-cy$ of Claim \ref {??}, we obtain a contradiction. This shows that in fact, there cannot exist  linearly independent sign-compatible vectors $x,y\in\mathcal N(S)$. Hence, we have proved Fact \ref{faux?}.
\epf

Now we can prove Lemma \ref{finitesupport}. Towards a contradiction, assume that $\mathcal N(S)\cap S_{E_N}$ is infinite. Consider the map $\Phi:\mathcal N(S)\cap S_{E_N}\to \{ -1,0,1\}^{N+M+2}$ defined by 
\[ \Phi(u):= \bigl(\sgn(\langle e_0^*,u\rangle),\dots ,\sgn(\langle e_N^*,u\rangle), \sgn(\langle e_0^*, Su\rangle),\dots ,\sgn(\langle e_M^*, Su\rangle) \bigr),\]
where $\sgn(0)=0$. Since $\Phi$ has finite range and $\mathcal N(S)\cap S_{E_N}$ is assumed to be infinite, $\Phi$ cannot be one-to-one. So one can find $x,y\in\mathcal N(S)\cap S_{E_N}$ such that $\Phi(x)=\Phi(y)$. Then $x$ and $y$ are sign compatible; and they are also linearly independent because $\Phi(-u)\neq \Phi(u)$ for every $u\in \mathcal N(S)\cap S_{E_N}$. This contradicts Fact \ref{faux?}. 
\epf

\subsection{Special operators} The next definition introduces a somewhat strange-looking class of operators, which will nonetheless be quite useful to us.

\begin{definition} Let $N\in\Z_+$, and let $S\in\mathcal B(E_N, \ell_p)$. We say that $S$ is a \emph{special operator} if there exists $R\in\Z_+$ such that 
\[ \forall n\in \llbracket 0,N\rrbracket\;:\; (I-P_R) Se_n=\alpha_n \sum_{j\in \Lambda_n} e_j,\]
where $\alpha_0,\dots , \alpha_N$ are non-zero scalars and $\Lambda_0,\dots ,\Lambda_N$ are finite subsets of $\Z_+$ with $\Lambda_n\cap \llbracket 0,R\rrbracket=\emptyset$ which are ``pairwise intersecting but not 3$\,$-$\,$by$\,$-$\,$3 intersecting'', \mbox{\it i.e.} $\Lambda_n\cap \Lambda_{n'}\neq \emptyset$ for any $n,n'$ but 
$\Lambda_n\cap \Lambda_{n'}\cap \Lambda_{n''}=\emptyset$ for any pairwise distinct $n,n', n''$.
\end{definition}

\blm\label{specapprox} Let $A\in\mathcal B_1(\ell_p)$, and let $N_0\in\Z_+$. For any $\varepsilon >0$, one can find a special operator $S\in\mathcal B(E_{N_0+1}, \ell_p)$ such that $\Vert S\Vert =1$ and $\Vert Se_n-Ae_n\Vert <\varepsilon$ for all $n\in\llbracket 0,N_0\rrbracket$.
\elm
\bpf Choose $R\in\Z_+$ such that $\Vert P_RAe_n- Ae_n\Vert <\varepsilon/2$ for all $n\in\llbracket 0,N_0\rrbracket$. Choose also finite sets $\Lambda_0,\dots ,\Lambda_{N_0+1}\subseteq \Z_+$ with $\Lambda_n\cap \llbracket 0,R\rrbracket=\emptyset$, pairwise but not 3$\,$-$\,$by$\,$-$\,$3 intersecting.  Finally, let $\alpha, \beta >0$, and define $S=S_{\alpha,\beta}\in \mathcal B(E_{N_0+1}, \ell_p)$ as follows:

\[ \left\{ \begin{array}{ll} \displaystyle Se_n:= (1-\varepsilon/2) P_R Ae_n +\alpha \sum_{j\in \Lambda_n} e_j&\quad\hbox{for $n=0,\dots ,N_0$},\\
\displaystyle Se_{N_0+1}:=\beta \sum_{j\in \Lambda_{N_0+1}} e_j.
\end{array}
\right.
\]
Then $S$ is a special operator by definition, and if $\alpha$ is small enough we have $\Vert SP_{N_0}\Vert< 1$ and $\Vert Se_n-Ae_n\Vert <\varepsilon$ for all $n\in\llb 0,N_0\rrb$. Having fixed $\alpha$ small enough, we may then choose $\beta$ in such a way that $\Vert S\Vert=1$. (Take the largest $\beta$ such that $\Vert S_{\alpha,\beta}\Vert\leq 1$, which exists since $\Vert S_{\alpha,\beta}\Vert\to\infty$ as $\beta\to\infty$.)
\epf

\begin{remark*} In order to construct the pairwise intersecting but not 3$\,$-$\,$by$\,$-$\,$3 intersecting sets $\Lambda_0$, $\dots$, $\Lambda_{N_0+1}$, one may proceed as follows. Denote by $\mathcal F$ the family of all $2$-elements subsets of $\llb 0, N_0+1\rrb$; choose a family $(k_{\{ n,n'\}})_{\{ n,n'\}\in\mathcal F}$ of pairwise distinct integers greater than $R$; and set $\Lambda_n:=\bigl\{ k_{\{ n,n'\}};\; n'\neq n\bigr\}$ for $n=0,\dots ,N_0+1$. 
\end{remark*}

\smallskip
\blm\label{suppspec} Let $p>2$. If $S\in \mathcal B(E_N,\ell_p)$ is a special operator, then every norming vector $x$ for $S$ has full support, \mbox{\it i.e.} ${\rm supp}(x)=\llbracket 0,N\rrbracket$. In particular, $S$ has the interval property.
\elm
\bpf Let $S\in \mathcal B(E_N,\ell_p)$ be a special operator with witnesses $\Lambda_n,\alpha_n, \beta_n$, and let $x\in\mathcal N(S)$. Towards a contradiction, assume that $\langle e_n^*,x\rangle=0$ for some $n\in \llbracket 0,N\rrbracket$. Then $e_n$ and $x$ have disjoint supports. Since $p>2$ and $x\in \mathcal N(S)$, it follows that $Se_n$ and $Sx$ have disjoint supports as well, by Lemma \ref{Kansupp}.

Since $x\neq 0$ we may fix $n'\in \llbracket 0,N\rrbracket$ such that $\langle e_{n'}^*,x\rangle\neq 0$; and by assumption on the sets $\Lambda_0,\dots ,\Lambda_N$, we may choose $j\in \Lambda_n\cap \Lambda_{n'}$. Then $j\not\in\Lambda_{n''}$ for any $n''\neq n, n'$; hence
\[ \langle e_j^*, Sx\rangle=\alpha_n \langle e_n^*,x\rangle + \alpha_{n'} \langle e_{n'}^*, x\rangle.\]
However, we have $\langle e_n^*,x\rangle=0$, and also $\langle e_j^*, Sx\rangle=0$ because $j\in \Lambda_n$ (so that $j\in {\rm supp}(Se_n)$). So we get $ \langle e_{n'}^*, x\rangle=0$, which is the required contradiction.
\epf

\subsection{Modifications} The next definition describes a way of ``modifying'' a finite-dimensional operator without losing any  norming vector and with in mind the goal of creating new ones.
\bdf\label{modif} Le $N,M\in\Z_+$, and let $S\in\mathcal B(E_N,E_M)$ be such that $\Vert S\Vert= 1$ and $\sp\bigl(\mathcal N(S)\bigr)\neq E_N$. Denote by  $\mathcal N(S)^\perp\subseteq E_N^*$ the annihilator of $\mathcal N(S)$, and let $\mathbf y=(y_1^*,\dots ,y_L^*)$ be a finite sequence of non-zero vectors in $\mathcal N(S)^\perp$. Let also $\eta \in (0, (2L)^{-1/p})$ and $\delta>0$. The \emph{$(\mathbf y, \eta,\delta)\,$-$\,$modification of $S$} is the operator $S_{\mathbf y,\eta,\delta}\in \mathcal B(E_{N+1}, E_{M+2L})$ defined as follows:
\[\left\{ 
\begin{array}{lc}
\displaystyle
S_{\mathbf y,\eta,\delta} x:= Sx+\delta \sum_{l=1}^L \langle  y_l^*, x\rangle \, \bigl(e_{M+2l-1}+e_{M+2l}\bigr)& \hbox{for all $x\in E_N$},\\
\displaystyle S_{\mathbf y,\eta,\delta} e_{N+1}:= \eta \sum_{l=1}^L \bigl(e_{M+2l-1}-e_{M+2l}\bigr).
\end{array}\right.\]
\edf

\smallskip 
Thus, in matrix form,
\[S_{\mathbf y, \eta,\delta}=
\left[
\begin{array}{ccc}

\begin{array}{|c|}
\hline \\
\begin{array}{cccccccc}
& \phantom{a}& & &  && \phantom{a}&\\
&  \phantom{a}& & &  && \phantom{a}&\\
&  \phantom{a}& & &  && \phantom{a}&\\
& \phantom{a}&&{\hbox{ \fontsize{18}{18}\selectfont $S$}}& &&  \phantom{a}& \\
&\phantom{a}& & & && \phantom{a}&\\
&  \phantom{a}& & &&& \phantom{a}& \\
&  \phantom{a}& & &  && \phantom{a}&\\
\\
\end{array}\\
\hline
\end{array}

&

\begin{array}{c}
\;\; 0 \\
\;\; \vdots\\
\;\; \vdots\\
\;\; \vdots\\
\;\; \vdots\\
\;\; \vdots\\
 \;\; 0

\end{array}

\\

\begin{array}{ccc}
\\

\delta\,\langle y_1^*,e_0\rangle & \cdots& \delta\,\langle y_1^*,e_N\rangle\\
\delta\,\langle y_1^*,e_0\rangle & \cdots& \delta\,\langle y_1^*,e_N\rangle\\
\vdots &&\vdots \\
\vdots& &\vdots \\
\vdots& &\vdots \\
\delta\, \langle y_L^*,e_0\rangle & \cdots& \delta\,\langle y_L^*,e_N\rangle\\
\delta\,\langle y_L^*,e_0\rangle & \cdots& \delta\,\langle y_L^*,e_N\rangle\end{array}

&

\begin{array}{c}
\\
\;\;\eta \\
-\eta \\
\;\;\vdots\\
\;\;\vdots\\
\;\;\vdots\\
\;\;\eta\\
-\eta \\
\end{array}

\end{array}
\right]
\]

\medskip

The following lemma is obvious.

\par\smallskip
\blm\label{trivialtruc} With the notation of Definition \ref{modif}, we have $\Vert S_{\mathbf y,\eta, \delta}x\Vert\geq \Vert Sx\Vert$ for all $x\in E_N$ and $S_{\mathbf y, \eta, \delta}x=Sx$ for all $x\in \mathcal N(S)$. In particular $\Vert S_{\mathbf y, \eta,\delta}\Vert\geq \Vert S\Vert=1$, and if $\Vert S_{\mathbf y, \eta,\delta}\Vert=1$ then $\mathcal N(S)\subseteq \mathcal N(S_{\mathbf y, \eta,\delta})$.
\elm

\smallskip Note that, given $S$, $\mathbf y$ and $\eta$,  it is not at all clear that one can find $\delta >0$ such that $\Vert S_{\mathbf y, \eta,\delta}\Vert=1$. This motivates the following definition.

\bdf We will say that $p$ is \emph{good exponent} if the following holds true: for any $S\in\mathcal B(E_N,E_M)$ such that $\Vert S\Vert =1$ and $\mathcal N(S)\cap S_{E_N}$ is a finite set, for any $\mathbf y=(y_1^*,\dots ,y_L^*)\subseteq\mathcal N(S)^\perp$ and for any $\eta\in (0, (2L)^{-1/p})$, one can find $\delta>0$ such that $\Vert S_{\mathbf y, \eta,\delta}\Vert= 1$.
\edf

\smallskip The next lemma is the most technical part in the proof of Proposition \ref{keymachin}. We can prove it for $p=3$ only, but it seems hard to believe that it does not hold true for all $p$.
\blm \label{hardtruc} $p:=3$ is a good exponent. 
\elm
\bpf  Let us fix $S\in \mathcal B(E_N,E_M)$ with $\Vert S\Vert =1$ such that $\mathcal N(S)\cap S_{E_N}$ is a finite set. For notational simplicity and since $\mathbf y$ and $\eta$ will remain fixed, we write $S_\delta$ instead of $S_{\mathbf y,\eta,\delta}$. For any $x\in\mathcal N(S)$ and $\varepsilon >0$, we set
\[ \Sigma(x,\varepsilon):=\bigl\{ u\in E_N;\; \Vert u\Vert=\varepsilon\;{\rm and}\; \Vert x+u\Vert =1\bigr\}.\]
Finally, recall the definition of sign-compatibility: two vectors $x,y\in E_N$ are said to be sign-compatible with respect to $S$ if  $\langle e_n^*, x\rangle\,\langle e_n^*,y\rangle\geq 0$ for all $n\in \llbracket 0,N\rrbracket$ and $\langle e_m^*, Sx\rangle\,\langle e_m^*,Sy\rangle\geq 0$ for all 
$m\in \llbracket 0,M\rrbracket$.

\smallskip

\begin{claim}\label{machin1} One can find $\varepsilon \in (0,1]$ such that, for any $x\in\mathcal N(S)\cap S_{E_N}$ and all $u\in\Sigma(x,\varepsilon)$, the following holds true: $x+u$ and $x$ are sign-compatible with respect to $S$, and $\Vert S(x+u)\Vert <\Vert x+u\Vert$.

\end{claim}
\begin{proof}[Proof of Claim \ref{machin1}] Since $\mathcal N(S)\cap S_{E_N}$ is a finite set, one can find $\varepsilon >0$ such that, for every $x\in \mathcal N(S)\cap S_{E_N}$, the set $\overline B(x,\varepsilon)\cap S_{E_N}$ contains no norming vector for $S$ except $x$. Then $\Vert S(x+u)\Vert <\Vert x+u\Vert$ for every $x\in \mathcal N(S)\cap S_{E_N}$ and all $u\in \Sigma (x,\varepsilon)$. Moreover, taking $\varepsilon$ small enough, we may also assume that the sign-compatibility condition is satisfied (and that $\varepsilon\leq 1$); see the beginning of the proof of Lemma \ref{sign}.
\epf

\begin{claim}\label{machin2} Let $\varepsilon$ be as in Claim \ref{machin1}. One can find $\delta_0>0$ such that, for any $0\leq \delta \leq \delta_0$, every $x\in\mathcal N(S)\cap S_{E_N}$ and all $z\in E_{N+1}$ such that $P_Nz\in \Sigma(x,\varepsilon)$, we have $\Vert S_\delta(x+z)\Vert\leq \Vert x+z\Vert$.
\end{claim}
\begin{proof}[Proof of Claim \ref{machin2}] Let 
\[ \kappa := \inf\,\bigl\{ \Vert x+u\Vert^p-\Vert S(x+u)\Vert^p;\; x\in \mathcal N(S)\cap S_{E_N}\,,\; u\in\Sigma(x,\varepsilon)\bigr\} \]
and note that, since $\mathcal N(S)\cap S_{E_N}$ is a finite set, $\kappa>0$ by compactness of the sets $\Sigma(x,\varepsilon)$.

\smallskip
Now, let us fix $x\in \mathcal N(S)\cap S_{E_N}$ and $z\in E_{N+1}$ such that $P_Nz\in \Sigma(x,\varepsilon)$. Write $z=u+\mu e_{N+1}$, where $u:=P_Nz$ and $\mu\in\R$. For any $\delta \geq 0$, the definition of $S_\delta=S_{\mathbf y, \eta,\delta}$ gives  
\begin{align*}
\Vert S_\delta(x+z)\Vert^p=\Vert S(x+u)\Vert^p+\sum_{l=1}^L \Bigl( \vert \delta \,\langle y_l^*, u\rangle +\eta \mu\vert^p+  \vert \delta \,\langle y_l^*, u\rangle - \eta \mu\vert^p\Bigr).
\end{align*}
(We have used the fact that $S_\delta x=Sx$.) Hence, setting $C:= \max(\Vert y_1^*\Vert, \dots , \Vert y_L^*\Vert)$, we get 
\begin{align*} \Vert S_\delta(x+z)\Vert^p&\leq \Vert S(x+u)\Vert^p+2 L\, \bigl( C\delta \Vert u\Vert +\eta \vert\mu\vert\bigr)^p\\
&\leq \Vert S(x+u)\Vert^p+2 L\, \bigl( C\delta  +\eta \vert\mu\vert\bigr)^p\qquad\hbox{because $\Vert u\Vert =\varepsilon\leq 1$.}
\end{align*}

By the definition of $\kappa$, it follows that
\[ \Vert x+ z\Vert^p -\Vert S_\delta(x+z)\Vert^p \geq \kappa +\vert \mu\vert^p -2L\, \bigl( C\delta  +\eta \vert\mu\vert\bigr)^p.\]

So, if we introduce the function $f_\delta :\R_+\to \R$ defined by 
\[ f_\delta(s):= s^p -2L \bigl( C\delta  +\eta s)^p,\]
it is enough to show that if $\delta_0>0$ is small enough, then $f_\delta \geq -\kappa$ on $\R_+$ for every $0\le \delta\leq \delta_0$. 

Let $\alpha>0$ to be determined in a few lines. We choose $\delta_0$ such that 
\[ C\delta_0\leq \alpha\eta\qquad\hbox{and}\qquad (C\delta_0+\eta s)^p\leq \frac\kappa{2L} +(\eta s)^p\quad \hbox{for all $s\in [0, 1]$.}\] 

If $0\le \delta\leq \delta_0$ then, on the one hand,
 \[ f_\delta(s)\geq s^p-\kappa -2L\eta^p s^p\geq -\kappa\qquad\hbox{for all $s\in [0,1]$},\]
 since $2L\eta^p< 1$. 
 On the other hand, if $s\geq 1$ then $ C\delta \leq \alpha\eta s$, and hence 
 \[ f_\delta(s)\geq s^p-2L\, (1+\alpha)^p\eta^p\, s^p.\]
 
 \noindent So, if we take $\alpha$ small enough to ensure that $2L(1+\alpha)^p\eta^p\leq 1$, which is possible since $2L\eta^p<1$, then $\delta_0$ has the required property. 
\epf

\begin{claim}\label{machin3} Let $\varepsilon$ be as in Claim \ref{machin1}, and let $\delta_0$ be as in Claim \ref{machin2}. If $0\leq \delta\leq \delta_0$ then, for every $x\in\mathcal N(S)\cap S_{E_N}$ and all $z\in E_{N+1}$ such that $P_Nz\in \Sigma(x,\varepsilon)$, it holds that
\[ \forall t\in [0,1]\;:\; \Vert S_\delta (x+tz)\Vert \leq \Vert x+tz\Vert. \]
\end{claim}
\begin{proof}[Proof of Claim \ref{machin3}] 
Let us fix $0\leq \delta\leq\delta_0$, $x\in \mathcal N(S)\cap S_{E_N}$ and $z\in E_{N+1}$ such that $u:= P_Nz\in \Sigma(x,\varepsilon)$.  Define $\varphi:\R\to\R$ by 
\[ \varphi(t):= \Vert x+tz\Vert^3 -\Vert S_{\delta}(x+tz)\Vert^3.\]
We are going to show that $\varphi(t)\geq 0$ for all $t\in [0,1]$.

\smallskip Since $x\in\mathcal N(S)$, we have $S_\delta x=Sx$. Looking at the definition of $S_\delta$, it follows that there exists some real number $\alpha$ such that 
\begin{align*} \forall t\in\R\;:\; \varphi(t)&= \Vert x+tu\Vert^3 -\Vert S(x+tu)\Vert^3 + \alpha \,\vert t\vert^3\\
&=: \psi(t)+ \alpha \,\vert t\vert^3.\end{align*}

Now, since $x$ and $x+u$ are sign-compatible (with respect to $S$) by the choice of $\varepsilon$, $x$ and $x+tu$ are sign-compatible for every $t\in [0,1]$ by convexity. Writing down the definition of $\psi(t)$, it follows (as in the proof of Lemma \ref{sign}) that the function $\psi$ is polynomial (of degree at most $3$) on $[0,1]$; and since $\psi(t)\geq 0$ on $\R$ with $\psi(0)=0=\psi'(0)$ (because $x$ is a norming vector for $S$ and $\Vert S\Vert=1$), it must have the form $\psi(t)= at^2 +b t^3$, with $a\geq 0$. So we see that 
\[ \forall t\in [0,1]\;:\; \varphi(t) =a t^2 +c t^3,\]
for some constants $a\geq 0$ and $c\in\R$. 

Moreover, by Claim \ref{machin2} we know that $\varphi (1)= \Vert x+z\Vert^3-\Vert S_\delta(x+z)\Vert^3\geq 0$, \mbox{\it i.e.} $a+c\geq 0$. So we get $\varphi(t) =at^2+ct^3\geq (a+c)t^3\geq 0$ for all $t\in [0,1]$, as required. 
\epf

\begin{claim}\label{machin4} Let $\varepsilon$ and $\delta_0$ be as in Claim \ref{machin2}, and let $0\leq\delta\leq\delta_0$. If $x\in\mathcal N(S)\cap S_{E_N}$, then one can find $r>0$ such that $\Vert S_\delta v\Vert \leq \Vert v\Vert$ for all $v\in E_{N+1}$ satisfying 
$\Vert P_Nv-x\Vert<r$.
\end{claim}
\begin{proof}[Proof of Claim \ref{machin4}] Let us fix $x\in\mathcal N(S)\cap S_{E_N}$. Let also $r>0$, and let $v\in E_{N+1}$ be such that 
$\Vert P_Nv-x\Vert<r$. Write $v= u+\mu e_{N+1}$ with $u:=P_Nv$ and $\mu\in\R$. 

\smallskip If $u\in\R x$, then $S_\delta u=Su$ and hence $\Vert S_\delta v\Vert^p= \Vert Su\Vert^p+ 2L \eta^p \,\vert\mu\vert^p\leq \Vert u\Vert^p+\vert \mu\vert^p=\Vert v\Vert^p$ because $2L\eta^p\leq 1$. So we assume that $u\not\in \R x$. By Claim \ref{machin3}, it is enough to show that one can find $\lambda\in\R$, $z\in P_N^{-1}\bigl(\Sigma(x,\varepsilon)\bigr)$ and $t\in [0,1]$ such that $v=\lambda (x+tz)$. 

\smallskip Consider the $2$-dimensional space $F:=\sp(x, u)\subseteq E_N$, its unit sphere $S_F$ and the sphere $S_F(x,\varepsilon)=\{w\in F\,;\,\Vert w-x\Vert =\varepsilon\}\subseteq F$. Since $x\in S_F$ and $\Vert u-x\Vert <r$, it is geometrically clear that if $r$ is small enough, then one can find $w\in S_F(x,\varepsilon)\cap S_F$  such that the half-line $\R_+ u$ intersects the half-open segment $(x,w]$. 
This means that one can find $\alpha >0$, $\widetilde u\in \Sigma(x,\varepsilon)$ and $t\in (0,1]$ such that $\alpha u= x+t\widetilde u$.  Since $v=u+\mu e_{N+1}$, it follows that one can find $\lambda\in\R$ and $z\in P_N^{-1}(\Sigma (x,\varepsilon))$ such that $v =\lambda (x+t z)$; explicitly, $v= (1/\alpha) \bigl( x+ t\, (\widetilde u+ (\alpha\mu/t) e_{N+1}\bigr)\bigr)$.
\epf

We can now prove Lemma \ref{hardtruc}. Let $\delta_0$ be as in Claim \ref{machin2}. By Claim \ref{machin4}, one can find an open set $V\subseteq E_{N}$ containing $\mathcal N(S)\cap S_{E_N}$ such that 
\[ \forall \delta\in [0,\delta_0]\;\; \forall v\in P_N^{-1}(V)\;:\; \Vert S_\delta v\Vert \leq \Vert v\Vert.\]
Note that, by homogeneity, the inequality $\Vert S_\delta v\Vert \leq \Vert v\Vert$ holds in fact for every   $v\in P_N^{-1}\bigl( \widehat V\bigr)$ where $\widehat V:=\bigcup_{\lambda\in\R} \lambda V$.

Let $K:=S_{E_{N}}\setminus \widehat V$ and $\widehat K:=\bigcup_{\lambda\in\R} \lambda K$. 
Since $V\supseteq \mathcal N(S)\cap S_{E_N}$, we have $\Vert Su\Vert <1$ for all $u\in K$. Since $K$ is compact, one can find $c <1$ such that $\Vert S u\Vert\leq c$ for all $u\in K$. Then $\Vert S  u\Vert \leq c\,\Vert u\Vert$ for all $u\in \widehat K$ by homogeneity. We claim that if $\delta$ is small enough, then
\[ \forall v\in P_N^{-1}(\widehat K)\;:\; \Vert S_\delta v\Vert \leq \Vert v\Vert.\]
Indeed, let $v\in P_{N}^{-1}(\widehat K)$ and write $v=u+\mu e_{N+1}$ where $u:= P_Nv\in \widehat K$. Then 
\begin{align*} \Vert S_\delta(v)\Vert &\leq \Vert Su+S_\delta(\mu e_{N+1})\Vert +\Vert (I-P_M) S_\delta u\Vert\\
&\leq \Vert Su+S_\delta(\mu e_{N+1})\Vert +C\delta \Vert u\Vert,
\end{align*}
where $C:= (2L)^{1/p} \max(\Vert y_1^*\Vert,\dots ,\Vert y_L^*\Vert)$. Moreover,
\begin{align*}
\Vert Su+S_\delta(\mu e_{N+1})\Vert^p&=\Vert Su\Vert^p+2L\eta^p \vert\mu\vert^p\\
&\leq \theta^p (\Vert u\Vert^p +\vert\mu\vert^p)=\theta^p \Vert v\Vert^p,
\end{align*}
where $\theta:=\max(c, (2L)^{1/p}\eta)<1$. So, if we choose $\delta$ such that $\theta +C\delta\leq 1$, we get that $\Vert S_\delta v\Vert\leq \Vert v\Vert$ for all $v\in P_N^{-1}(\widehat K)$.

Altogether, we have shown that if $\delta>0$ is small enough, then $\Vert S_\delta v\Vert\leq \Vert v\Vert$ for all $v\in P_N^{-1}(\widehat V\cup\widehat K)$. Since $\widehat K\cup \widehat V$ contains $B_{E_N}$, this implies that $\Vert S_\delta\Vert\leq 1$, which concludes the proof.
\epf

\smallskip
\blm\label{dim+2} Let $p$ be a good exponent. Let $S\in\mathcal B(E_N,E_M)$ and $\mathbf y=(y_1^*,\dots ,y_L^*)$ be as in Definition \ref{modif}, and assume that $\mathcal N(S)\cap S_{E_N}$ is a finite set. Let also $\eta\in (0, (2L)^{-1/p})$. \be
\item[\rm (i)] There exists a largest $\delta>0$ such that $\Vert S_{\mathbf y, \eta,\delta}\Vert= 1$.
\item[\rm (ii)] 
If $\delta$ is as in \emph{(i)}, then $\dim\,\sp\bigl( \mathcal N(S_{\mathbf y,\eta,\delta})\bigr)\geq 2+ \dim\sp\bigl(\mathcal N(S)\bigr)$.
\ee
\elm
\bpf (i) Since $p$ is a good exponent, there exists at least one $\delta >0$ such that $\Vert S_{\mathbf y, \eta,\delta}\Vert= 1$. Moreover, since the linear functionals $y_l^*$ are non-zero, it is clear that $\Vert S_{\mathbf y, \eta,\delta}\Vert\to\infty$ as $\delta\to\infty$. This proves (i).

\smallskip
(ii) This will follow from the next three claims, which will also be used in the proof of Lemma \ref{induction} below. Let $\delta$ be as in (i) and, for notational simplicity, let $\widetilde S:=S_{\mathbf y, \eta,\delta}$. 

\begin{claim}\label{newnorm}  $ \mathcal N(S)$ is strictly contained in $\mathcal N(\widetilde S)$.
\end{claim} 
\begin{proof}[Proof of Claim \ref{newnorm}] By Lemma \ref{trivialtruc}, we know that $\mathcal N(\widetilde S)\supseteq \mathcal N(S)$. Towards a contradiction, assume that $\mathcal N(\widetilde S)= \mathcal N(S)$. Then, extending the linear functionals $y_1^*,\dots ,y_L^*$ to $E_{N+1}$ by setting $\langle y_l^*, e_{N+1}\rangle:= 0$, we have $y_l^*\in\mathcal N(\widetilde S)^\perp$ for $l=1,\dots ,L$; so we may consider $(\mathbf y, \eta, \delta')\,$-$\,$modifications of $\widetilde S\in\mathcal B(E_{N+1}, E_{M+2L})$ for any $\delta'>0$. 

Since $p$ is a good exponent and $\mathcal N(\widetilde S)\cap S_{E_{N+1}}= \mathcal N(S)\cap S_{E_N}$ is a finite set, one can find $\varepsilon >0$ such that the operator $T:= \widetilde S_{\mathbf y, \eta, \varepsilon \delta}\in \mathcal B(E_{N+2}, E_{M+4L})$ satisfies $\Vert T\Vert =1$. Thus, for any $(x,\lambda, \mu)\in E_{N}\times\R\times \R$ we have \[ \Vert T(x+\lambda e_{N+1}+\mu e_{N+2})\Vert^p\leq \Vert x\Vert^p+\vert\lambda \vert^p+\vert\mu\vert^p.\]

By the definition of $T$ and since $\langle y_l^*, e_{N+1}\rangle=0$ for $l=1,\dots ,L$, this reads as follows:
\begin{align*}  \Vert \widetilde S (x+\lambda e_{N+1})\Vert^p+ \sum_{l=1}^L \Bigl( \left\vert \varepsilon\delta\, \langle y_l^*,x\rangle +\eta \mu\right\vert^p+ \left\vert \varepsilon\delta\, \langle y_l^*,x\rangle -\eta \mu\right\vert^p\Bigr)\leq \Vert x\Vert^p+\vert\lambda \vert^p+\vert\mu\vert^p.
\end{align*}

Taking $\mu:= \varepsilon\,\lambda$, we obtain that for all $(x,\lambda)\in E_{N}\times \R$,
\[  \Vert \widetilde S (x+\lambda e_{N+1})\Vert^p+ \varepsilon^p\sum_{l=1}^L \Bigl( \bigl\vert \delta\, \langle y_l^*,x\rangle +\eta\,  \lambda\bigr\vert^p+ \bigl\vert \delta\, \langle y_l^*,x\rangle -\eta \,\lambda\bigr\vert^p\Bigr) \leq \Vert x\Vert^p+(1+\varepsilon^p) \vert\lambda\vert^p.
\]

However, the definition of $\widetilde S=S_{\mathbf y, \eta,\delta}$ gives 
\[ \Vert \widetilde S(x+\lambda e_{N+1})\Vert^p=\Vert S x\Vert^p + \sum_{l=1}^L \Bigl( \bigl\vert \delta\, \langle y_l^*, x\rangle +\eta\,  \lambda\bigr\vert^p+ \bigl\vert \delta\, \langle y_l^*, x\rangle -\eta \,\lambda\bigr\vert^p\Bigr).\]

So we see that for any $(x,\lambda)\in E_N\times\R$,
\[ \Vert S x\Vert^p + (1+\varepsilon^p) \sum_{l=1}^L \Bigl( \bigl\vert \delta\, \langle y_l^*,x\rangle +\eta\,  \lambda\bigr\vert^p+ \bigl\vert \delta\, \langle y_l^*,x\rangle -\eta \,\lambda\bigr\vert^p\Bigr)
\leq \Vert x\Vert^p +(1+\varepsilon^p) \vert\lambda \vert^p
\]

Setting $\kappa:=(1+\varepsilon^p)^{1/p}$ and replacing $\lambda$ by $\lambda/\kappa$, this may be re-written as follows:
\[ \forall (x,\lambda)\in E_N\times\R\;:\; \Vert S x\Vert^p + \sum_{l=1}^L \Bigl( \bigl\vert \kappa \delta\, \langle y_l^*,x\rangle +\eta\,  \lambda\bigr\vert^p+ \bigl\vert \kappa \delta\, \langle y_l^*,x\rangle -\eta \,\lambda\bigr\vert^p\Bigr)
\leq \Vert x\Vert^p +  \vert\lambda \vert^p.
\]

This means exactly that the operator $S_{\mathbf y, \eta , \kappa\delta}$ satisfies $\Vert S_{\mathbf y, \eta , \kappa\delta}\Vert \leq 1$, which contradicts the maximality of $\delta$ since $\kappa>1$.
\epf

\begin{claim}\label{noNs} Let $z\in E_{N+1}$, and assume that $z\in\mathcal N(\widetilde S)\setminus \mathcal N(S)$. Then $\langle e_{N+1}^*,z\rangle \neq 0$ and $P_Nz\not\in\sp\bigl(\mathcal N(S)\bigr)$.
\end{claim}
\begin{proof}[Proof of Claim \ref{noNs}] Assume that $\langle e_{N+1}^*,z\rangle=0$. Then $z$ and  $e_{N+1}$ have disjoint supports. Since $z\in\mathcal N(\widetilde S)$ and $p>2$, it follows that $\widetilde Sz$ and $\widetilde Se_{N+1}$ also have disjoint supports, by Lemma \ref{Kansupp}. By the definition of $\widetilde S=S_{\mathbf y, \eta,\delta}$, we thus see that $\langle y_l^*,z\rangle=0$ for $l=1,\dots ,L$. Hence $\widetilde Sz=Sz$, which contradicts our assumption that $z\in\mathcal N(\widetilde S)\setminus \mathcal N(S)$ since $\Vert \widetilde S\Vert =1=\Vert S\Vert$. 
\par\smallskip
Assume now that $P_Nz\in \sp\bigl(\mathcal N(S)\bigr)$. Then by the definition of $\widetilde S$ we have
\[ \widetilde Sz=S(P_Nz)+ \eta \, \langle e_{N+1}^*, z\rangle\sum_{l=1}^L (e_{M+2l-1}-e_{M+2l}),\]
so that $\Vert \widetilde Sz\Vert^p= \Vert S P_Nz\Vert^p + 2L \eta^p\vert \langle e_{N+1}^*,z\rangle\vert^p$. Since $\langle e_{N+1}^*,z\rangle\neq 0$ and $2L\eta^p<1$, it follows that 
\[ \Vert \widetilde Sz\Vert^p<\Vert S P_Nz\Vert^p +\vert \langle e_{N+1}^*,z\rangle\vert^p\leq \Vert P_Nz\Vert^p+ \vert \langle e_{N+1}^*,z\rangle\vert^p =\Vert z\Vert^p,\]
which contradicts the fact that $z\in\mathcal N(\widetilde S)$.
\epf

\begin{claim}\label{dim2} Let $z\in \mathcal N(\widetilde S)$, and write $z$ as $z=x+\lambda e_{N+1}$ with $x\in E_N$ and $\lambda\in\R$. Then the vector $z':= x-\lambda e_{N+1}$ also belongs to $\mathcal N(\widetilde S)$.
\end{claim}
\begin{proof}[Proof of Claim \ref{dim2}] We have  $\Vert z'\Vert=\Vert z\Vert$, and it is clear from the definition of $\widetilde S$ that $\Vert \widetilde Sz'\Vert=\Vert \widetilde Sz\Vert$. 
\epf 

\begin{remark}\label{span} With the notation of Claim \ref{dim2}, if $\lambda\neq 0$ then $e_{N+1}\in\sp(z,z')$. So, from Claims \ref{newnorm} and \ref{noNs}, we see that $e_{N+1}\in \sp\bigl(\mathcal N(\widetilde S)\bigr)$. This remark will be used in the proof of Lemma \ref{induction} below.
\end{remark}

We can now prove (ii), \mbox{\it i.e.} that $\dim\,\sp\bigl( \mathcal N(S_{\mathbf y,\eta,\delta})\bigr)\geq 2+ \dim\sp\bigl(\mathcal N(S)\bigr)$. 
Let us fix $z\in \mathcal N(\widetilde S)\setminus \mathcal N(S)$ given by Claim \ref{newnorm}. Write $z=x+\lambda e_{N+1}$ where $x\in E_N$, and let $z':= x-\lambda e_{N+1}$, so that $z'\in \mathcal N(\widetilde S)$ by Claim \ref{dim2}. Since $x\neq 0$ and $\lambda\neq 0$ by Claim \ref{noNs},  $z$ and $z'$ are linearly independent. Hence, to conclude the proof of (ii), it is enough to show that $\sp(z,z')\cap \sp\bigl(\mathcal N(S)\bigr)=\{ 0\}$. 

Let $u\in \sp(z,z')\cap \sp\bigl(\mathcal N(S)\bigr)$. There exist $a,b\in\R$ such that $u=az+bz'=(a+b) x+\lambda (a-b) e_{N+1}$. Since $\lambda \neq 0$ and $u\in \sp\bigl(\mathcal N(S)\bigr)\subseteq E_N$, we must have $a-b=0$. Hence $u=2a x$; and since $u\in \sp\bigl(\mathcal N(S)\bigr)$ and $x\not\in \sp\bigl(\mathcal N(S)\bigr)$, it follows that $a=0$. So $u=0$, as required.
\epf

\blm Assume that $p>2$. Let $N_0\in\Z_+$ and let $\varepsilon >0$. Let also $N,M\in\Z_+$ with $N\geq N_0$, and let $S\in\mathcal B(E_N,E_M)$ and $\mathbf y=(y_1^*,\dots ,y_L^*)$ be as in Definition \ref{modif}. If $\eta\in (0, (2L)^{-1/p})$ is close 
enough to $(2L)^{-1/p}$ {\rm (}depending only on $\varepsilon$ and $\mathbf y${\rm )} then, for any $\delta>0$ such that $\Vert S_{\mathbf y, \eta,\delta}\Vert= 1$, it holds that $\Vert S_{\mathbf y,\eta,\delta} e_n-Se_n\Vert <\varepsilon$ for all $n\in\llb0,N_0\rrb$.
\elm
\bpf For any $\delta>0$ and any $n\in \llb 0,N_0\rrb$, we have 
\[ \Vert S_{\mathbf y,\eta,\delta} e_n-Se_n\Vert^p = 2\left(\sum_{l=1}^L \vert \langle y_l^*, e_n\rangle\vert^p\right) \delta.\]
So it is enough to show that if we choose $\eta$ sufficiently close to $(2L)^{-1/p}$, then any $\delta$ such that $\Vert S_{\mathbf y, \eta,\delta}\Vert=1$ must be very small.

\smallskip Let $c_l:= \langle y_l^*, e_0\rangle$ for $l=1,\dots ,L$. For any $\delta>0$ such that $\Vert S_{\mathbf y, \eta,\delta}\Vert=1$ and for every $t>0$, we have 
\begin{align*} 1+t^p= \Vert t e_0+e_{N+1}\Vert^p&\geq \Vert S_{\mathbf y,\eta,\delta} (t e_0 +e_{N+1}) \Vert^p\\
&\geq   \sum_{l=1}^L \Bigl( \vert  \delta c_l t+\eta\vert^p +\vert \delta c_l t -\eta\vert^p\Bigr).
\end{align*}

Now, since $p>2$, we can make use of Lemma \ref{KanIneq}: for any $u,v\in \C$,
\[ \vert u+v\vert^p+\vert u-v\vert^p\geq 2\vert u\vert^p +p \,\vert u\vert^{p-2} \vert v \vert^2.\]
Taking $u:= \eta$ and $v:= \delta c_lt$ for $l=1,\dots ,L$, this gives
\[2L \eta^p +C  \eta^{p-2} \delta^{2} t^2\leq 1+t^p\;,\qquad\hbox{where}\quad C:= p\sum_{l=1}^L \vert c_l\vert^2.\] 

Thus, writing $2L\eta^p =1-\alpha$ and assuming, as we may, that $\alpha\leq 1/2$, we obtain
\[ \forall t>0\;:\; K\delta^2\leq \frac\alpha{t^2}+ +t^{p-2},\]
where 
$K:= C\times (4L)^{-1+\frac2p} $. Optimising with respect to $t$, \mbox{\it i.e.} taking $t:=\left(\frac{2\alpha}{p-2}\right)^{1/p}$, this gives 
\[ K\delta^2\leq C_p \, \alpha^{1-\frac2p}\]
for some constant $C_p$ depending only on $p$. Hence we see that $\delta$ is indeed very small if $\eta$ is close enough to $(2L)^{-1/p}$.
\epf

\subsection{Maximal modifications} The following definition looks rather natural.

\bdf Let $N,M\in\Z_+$ and let $S\in \mathcal B(E_N,E_M)$ with $\Vert S\Vert=1$. We say that an operator $\widetilde S\in\mathcal B(E_{N+1},\ell_p)$ is a \emph{maximal modification of $S$} if $\widetilde S$ is a $(\mathbf y, \eta,\delta)\,$-$\,$modification of $S$ for some triple $(\mathbf y,\eta,\delta)$ such that $\mathbf y$ is a basis of $\mathcal N(S)^\perp$.
\edf

\smallskip The interest of this definition lies in the fact that, roughly speaking, maximal modifications ``preserve the interval property''. This is the content of the next lemma.

\blm\label{induction} Assume that $p>2$. Let $N, M\in\Z_+$, and let $(S_0,\dots , S_K)$ be a finite sequence of operators, with $S_k\in\mathcal B(E_{N+k}, E_{M +L_k})$ for some $L_k\in\Z_+$ {\rm (}with $L_0=0${\rm )} and $\Vert S_k\Vert =1$, such that $S_{k+1}$ is a maximal modification of $S_k$ for every $k<K$. If $S_0$ is a special operator, then each $S_k$ has the interval property.
\elm
\bpf Since $S_0$ is a special operator, any $z\in\mathcal N(S_0)$ has support equal to $\llb 0,N\rrb$ by Lemma \ref{suppspec}. In particular,  $S_0$ has the interval property. 

\smallskip Let us first show that $S_1$ has the interval property. Let $z\in E_{N+1}$ be a norming vector for $S_1$. We have to show that ${\rm supp}(x)=\llb0,a\rrb$ for some $a\leq N+1$.  If $z\in \mathcal N(S_0)$, then ${\rm supp}(z)=\llb 0,N\rrb$ because $S_0$ is a special operator. So we assume that $z\not\in \mathcal N(S_0)$. Then $\langle e_{N+1}^*,z\rangle \neq 0$ and $P_N z\neq 0$ by Claim \ref{noNs}. We show that ${\rm supp}(z)=\llb 0,N+1\rrb$.

Choose $n\in \llb 0,N\rrb$ such that $\langle e_{n}^*,z\rangle \neq 0$ and, towards a contradiction, assume that there exists $n'\in\llb 0,N\rrb$ such that $\langle e_{n'}^*,z\rangle = 0$. Since $S_0\in\mathcal B(E_N,E_M)$ is a special operator, one can find $m\in\llb 0,M\rrb$ such that $m\in{\rm supp}(S_0 e_n)\cap {\rm supp}(S_0 e_{n'})$ but $m\not\in {\rm supp}(S_0 e_{n''})$ for any $n''\in\llb 0,N\rrb\setminus\{ n,n'\}$. Since $S_1$, being a modification of $S_0$, satisfies $P_{M} S_1P_N=S_0$ and $P_{M} S_1e_{N+1}=0$, we see that $m\in {\rm supp}(S_1 e_n)\cap {\rm supp}(S_1 e_{n'})$ and $m\not\in {\rm supp}(S_1 e_{n''})$ for any $n''\in \llb 0,N+1\rrb\setminus\{ n,n'\}$. Since $\langle e_{n}^*,z\rangle \neq 0$ and $\langle e_{n'}^*,z\rangle = 0$, it follows that $\langle e_m^*, S_1z\rangle\neq 0$, \mbox{\it i.e.} $m\in{\rm supp}(S_1z)$. However, since $z$ and $e_{n'}$  have disjoint supports and $z\in\mathcal N(S_1)$, the vectors $S_1e_{n'}$ and $S_1 z$ must have disjoint supports by Lemma \ref{Kansupp}. This is the required contradiction.

\smallskip Now, assume that we have been able to prove that $S_0,\dots ,S_k$ have the interval property for some $1\leq k<K$, and let us show that $S_{k+1}$ has the interval property. Let $z\in E_{N+k+1}$ be a norming vector for $S_{k+1}$. We have to show that ${\rm supp}(z)$ is an interval $\llb 0, L\rrb$. If $z\in \mathcal N(S_k)$, then we are done by our induction hypothesis; so we assume that $z\not\in \mathcal N(S_k)$. Then $\langle e_{N+k+1}^*, z\rangle\neq 0$ by Claim \ref{noNs}. Hence, we have to show that $\langle e_n^*, z\rangle \neq 0$ for all $n\in\llb 0 , N+k\rrb$.

Let us first consider $n$ such that $N<n\leq N+k$, and write $n=N+k'$ with $1\leq k'\leq k$. Towards a contradiction, assume that $\langle e_{N+k'}^*,z\rangle=0$. Then ${\rm supp}(S_{k+1} e_{N+k'})\cap {\rm supp}(S_{k+1}z)=\emptyset$, by Lemma \ref{Kansupp}. In particular, since ${\rm supp}(S_{k'} e_{N+k'})\subseteq {\rm supp}(S_{k+1} e_{N+k'})$ (because $S_{k+1}$ is obtained from $S_{k'}$ by successive modifications), we have ${\rm supp}(S_{k+1}z) \cap {\rm supp}(S_{k'} e_{N+k'})=\emptyset$. Denoting by $\mathbf y=(y_1,\dots ,y_L^*)\subseteq \mathcal N(S_{k'-1})^\perp$ the sequence of linear functionals used to construct $S_{k'}$ from $S_{k'-1}$, this means that $\langle y_l^*, P_{N+k'-1}z\rangle=0$ for $l=1,\dots ,L$; and since $S_{k'}$ is a \emph{maximal} modification of $S_{k'-1}$, \mbox{\it i.e.} $\mathbf y$ is a basis of $\mathcal N(S_{k'-1})^\perp$, it follows that $P_{N+k'-1} z\in \sp\bigl( \mathcal N( S_{k'-1})\bigr)$. In particular, $P_{N+k'-1} z\in \sp\bigl(\mathcal N(S_k)\bigr)$. Moreover, by successive applications of Remark \ref{span}, we see that $e_{N+r}\in\sp\bigl( \mathcal N(S_r)\bigr)$ for all $1\leq r\leq K$. In particular, $e_{N+r}\in\sp\bigl( \mathcal N(S_k)\bigr)$ for all $k'\leq r\leq k$. So we get that $P_{N+k}z\in \sp\bigl( \mathcal N(S_k)\bigr)$, which contradicts Claim \ref{noNs} since $z\in\mathcal N(S_{k+1})\setminus \mathcal N(S_k)$.

It remains to consider the case where $n\in\llb 0,N\rrb$. Using the fact that $S_0$ is a special operator, one shows that $\langle e_n,z\rangle\neq 0$ exactly as in the above proof for $S_1$.
\epf

We are finally ready to prove Proposition \ref{keymachin}.

\subsection{Proof of Proposition \ref{keymachin}} Assume that $p=3$. We have to show that for any $A\in\mathcal B_1(\ell_p)$, for any $N_0\in\Z_+$ and any $\varepsilon >0$, one can find $T\in\mathcal B_1(\ell_p)$ such that $\overline{\sp}\bigl(\mathcal N(T)\bigr)=\ell_p$ and $\Vert Te_n-Ae_n\Vert <\varepsilon$ for all $n\in\llb 0,N_0\rrb$. Let us fix $A$, $N_0$ and $\varepsilon$.

\smallskip We only need to show that there exist $N', M'\in\Z_+$ with $N'\ge N_0$ and $S\in\mathcal B(E_{N'}, E_{M'})$ with $\Vert S\Vert=1$ such that $\sp\bigl( \mathcal N(S)\bigr)=E_{N'}$ and  $\Vert Se_n-Ae_n\Vert <\varepsilon$ for all $n\in\llb 0,N_0\rrb$. Indeed, if  this is has been done then, denoting by $\mathbf S$  the canonical forward shift on $\ell_p$, the operator $T:= S+ \mathbf S^{M'} (I-P_{N'})$ does the job because $\sp\bigl( \mathcal N(T)\bigr)$ contains $e_n$ for all $n\in\Z_+$. 

\smallskip
Let $N:= N_0+1$. By Lemma \ref{specapprox}, one can find $M\in\Z_+$ and a special operator $S_0\in\mathcal B(E_N,E_M)$ such that $\Vert S_0\Vert=1$ and $\Vert S_0e_n-Ae_n\Vert <\varepsilon/2$ for all $n\in\llb 0,N_0\rrb$. Let $d_0:= \dim \sp\bigl( \mathcal N(S_0)\bigr)$. If $d_0=N+1$, \mbox{\it i.e.} $\sp\bigl( \mathcal N(S_0)\bigr)= E_N$, we are done. Otherwise, by Lemmas \ref{suppspec}, \ref{finitesupport}, \ref{hardtruc} and \ref{dim+2}, one can find $S_1\in\mathcal B(E_{N+1}, M_1)$, a maximal modification of $S_0$, such that $\Vert S_1e_n-S_0e_n\Vert <\varepsilon/2^2$ for $n=0,\dots ,N_0$ and $\dim\sp\bigl(\mathcal N(S_1)\bigr)\geq d_0+2$. If $d_0+2=N+2=\dim E_{N+1}$, we are done. Otherwise, we can repeat the process (since $S_1$ has the interval property by Lemma \ref{induction} and hence $\mathcal N(S_1)\cap S_{E_{N+1}}$ is finite by Lemma \ref{finitesupport}), and continue as long as possible. In other words, as long as this is possible, we construct operators $S_k\in\mathcal B(E_{N+k}, M_k)$, $k\geq 0$ such that $S_{k+1}$ is a maximal modification of $S_k$ with $\Vert S_{k+1}e_n -S_ke_n\Vert  <\varepsilon/2^{k+1}$ for $n=0,\dots ,N_0$ and $\dim \sp\bigl(\mathcal N(S_k)\bigr)\geq d_0+2k$. Since at each step the dimension of $\sp\bigl( \mathcal N(S_k)\bigr)$ is increased at least by $2$ but the dimension of $E_{N+k}$ is only increased by $1$, the process must end after finitely many steps: there exists $K$ such that $\dim\sp\bigl( \mathcal N(S_K)\bigr)\geq \dim E_{N+K}$, and hence $\sp\bigl( \mathcal N(S_K)\bigr)=E_{N+K}$. Setting $N':=N+K$, $M':=M_K$ and $S:=S_K$, we have $\Vert Se_n-Ae_n\Vert < \varepsilon\sum_{k=0}^K 2^{-k-1}<\varepsilon$ for $n=0,\dots ,N_0$, which concludes the proof.
\hfill$\square$

\subsection{A remark concerning the case $p=4$} In this last sub-section, we would like to point out the following fact, which is in the spirit of Lemma \ref{sign}. 

\begin{fact}\label{p=4} Assume that $p=4$. Let $E$ be a subspace of $\ell_p$, and let $S\in\mathcal B(E,\ell_p)$. If $x,y\in\mathcal N(S)$ and if $\sp(x,y)$ contains a vector $z\in\mathcal N(S)$ which is not a multiple of $x$ or $y$, then $\sp(x,y)\subseteq\mathcal N(S)\cup\{ 0\}$. In particular, if 
$\dim(E)=2$ and if $S$ admits 3 pairwise linearly independent norming vectors, then $S$ is a scalar multiple of an isometry.
\end{fact}
\bpf We may assume that $\Vert S\Vert=1$. Let $x,y\in\mathcal N(S)$, and let us consider the function $\varphi:\R^2\to \R$ defined by 
\[ \varphi(s,t):= \Vert sx+ty\Vert^4-\Vert S(sx+ty)\Vert^4.\]

Since $p=4$, this a polynomial function of degree at most $4$. Moreover, we have $\varphi(s,t)\geq 0$ on $\R^2$, and $\varphi(s,0)\equiv 0\equiv \varphi(0,t)$. We have to show that if there exists $(s_0, t_0)$ such that $s_0 t_0\neq 0$ and $\varphi(s_0,t_0)=0$, then $\varphi(s,t)=0$ for all $(s,t)\in\R^2$. 

\smallskip Since  $\varphi(s,0)=0=\varphi(0,t)$, we may write
\[ \varphi(s,t)=st\, \psi(s,t),\]
where $\psi$ is a polynomial function of degree at most $2$. Moreover since $\varphi(s,t)\geq 0$ on $\R^2$, the function $\varphi$ has a minimum at $(s,0)$ and at $(0,t)$ for any $s,t\in\R$. So we must have 
\[ \frac{\partial\varphi}{\partial t}(s,0)\equiv 0\equiv \frac{\partial\varphi}{\partial s}(0,t),\] \mbox{\it i.e.}
\[ s\psi(s,0)\equiv 0\equiv t\psi(0,t).\]

This means that $\psi(s,0)\equiv 0\equiv \psi(0,t)$, and hence that $\psi(s,t) =c \, st$ for some constant $c$. So we obtain
\[ \varphi(s,t) =c\, s^2 t^2,\]
and the result follows.
\epf

\begin{remark} The proof of Fact \ref{p=4} has shown the following. Assume that $p=4$,  let $E$ be a subspace of $\ell_p$ and let $S\in\mathcal B(E,\ell_p)$ with $\Vert S\Vert=1$. If $x,y\in\mathcal N(S)$ then, for all $s,t\in\R$:  
\[\Vert sx+ty\Vert^4-\Vert S(sx+ty)\Vert^4=\bigl(  \Vert x+y\Vert^4-\Vert S(x+y)\Vert^4\bigr) \, s^2t^2.\]
In particular, we see that $\Vert sx+ty\Vert^4-\Vert S(sx+ty)\Vert^4=\Vert sy+tx\Vert^4-\Vert S(sy+tx)\Vert^4$; this is a ``symmetry'' property that was not immediately apparent. It seems natural to wonder whether this symmetry property holds true for any even integer $p=2n$. Note that this is indeed the case for $p=2$, since a direct computation using the scalar product reveals that $\Vert sx+ty\Vert^2-\Vert S(sx+ty)\Vert^2=\bigl(  \Vert x+y\Vert^2-\Vert S(x+y)\Vert^2\bigr) \, st$.
\end{remark}

\begin{remark} It is shown in \cite{Debm} that if $E=\ell_p^2$ and $p>1$ is arbitrary, then any $S\in \mathcal B(E)$ having 3 pairwise linearly independent norming vectors is a scalar multiple of an isometry. This might be true for any $2\,$-$\,$dimensional subspace of $\ell_p$. In fact, it is tempting to ``conjecture'' the following: if $E$ is a $d\,$-$\,$dimensional subspace of $\ell_p$ and if $S\in\mathcal B(E,\ell_p)$ admits $d+1$ norming vectors such that any $d$ of them are linearly independent, then $S$ is a scalar multiple of an isometry.
\end{remark}

\section{Typical properties of Hilbert space contractions}\label{SecThree}
In this section, we illustrate the point of continuity approach by giving short and streamlined proofs of two nice results due to respectively to Eisner \cite{E} and Eisner-M\'atrai \cite{EM} concerning Hilbert space contractions. 
We start with the following proposition, which is the main result of \cite{E}.

\bpr\label{Eisner} A typical $T\in \mathcal (B_1(\ell_2), \emph{\wot})$ is unitary.
\epr
\bpf This follows immediately from Corollary \ref{CPCom} and part (2) of  Proposition \ref{CPH}.
\epf

Now we turn to a rather surprising result from \cite{EM}. Let us denote by $B_\infty$ the backward shift with infinite multiplicity, \mbox{\it i.e.} the canonical backward shift acting on $\ell_2(\Z_+,\ell_2)$. 
\bth\label{Eisner-Matrai} A typical $T\in\bigl( \mathcal B_1(\ell_2),\emph{\sot}\bigr)$ is unitarily equivalent to $B_\infty$.
\eth

The proof relies on the next two lemmas, which we state in greater generality than what is needed for the proof of Theorem \ref{Eisner-Matrai}.

\blm\label{C0} If $X$ is any separable Banach space, then a typical $T\in (\bbx,\emph{\sot})$ is strongly stable, \mbox{\it i.e.} such that $T^n\xrightarrow{\emph{\sot}}0$ as $n\to\infty$.
\elm
\bpf This is proved in \cite[Lemma 5.9]{EM} or \cite[Proposition 3.7]{GMM2}. However, the proof is so short that we can reproduce it. Let $D$ be a countable dense subset of $X$. Then since $\Vert T^nx\Vert$ is non-increasing for any $T\in\bbx$ and any $x\in X$, an operator $T\in\bbx$ is strongly stable if and only if $\forall z\in D\;:\; \inf_{n\in\N} \Vert T^n z\Vert=0$. It follows that the set of strongly stable operators is $\gd$ in $(\bbx,\sot)$; and this set is also dense since it contains all $T\in\bbx$ such that $\Vert T\Vert<1$.
\epf
\par\smallskip
Recall that an operator $T\in\bx$ is said to be \emph{upper semi-Fredholm} if $T$ has closed range and $\dim\ker(T)<\infty$. Recall also that the Banach space $X$ has the \emph{Metric Approximation Property} if the identity operator $I_X$ can be uniformly approximated on compact sets by finite rank operators in $\bbx$.
\blm\label{Fredholm} Assume that the Banach space $X$ has the \emph{MAP}.
Then a typical $T\in(\bbx, \emph{\sot})$ is not upper semi-Fredholm.
\elm
\bpf 
We need the following fact, which is proved in \cite[Lemma 5.13]{EM} when $X$ is a Hilbert space.
\begin{fact}\label{EMn} A typical $T\in(\bbx,{\sot})$ has the following property: for any $\varepsilon>0$ and any $n\in\N$, there exists a subspace $E\subseteq X$ with $\dim(E)> n$ such that $\Vert T_{| E}\Vert<\varepsilon$. \end{fact}
\begin{proof}[\it Proof of Fact \ref{EMn}] For any fixed finite-dimensional subspace $E\subseteq X$ and any $\varepsilon >0$, the set $\{ T\in\bbx;\; \Vert T_{| E}\Vert<\varepsilon\}$ is easily seen to be \sot$\,$-$\,$open in $\bbx$. So the set of all $T\in\bbx$ with the above property is $\sot\,$-$\,\gd$. Moreover, this set is also \sot$\,$-$\,$dense in $\bbx$ since it contains all finite rank contraction operators (and the latter are dense in $\bbx$ thanks to the MAP).
\end{proof}

For any $n\in\Z_+$, let us set \[ \mathcal M_n:=\bigl\{ T\in \bbx;\; \hbox{$T$ has closed range and}\;\dim\ker(T)=n\bigr\}.\]

\noindent If $T\in\mathcal M_n$, then there is a subspace $F\subseteq X$ such that ${\rm codim}(F)=n$ and $T_{| F}$ is an embedding. By Fact \ref{EMn}, it follows that each set $\mathcal M_n$ is \sot$\,$-$\,$meager in $\bbx$; which proves the lemma.
\epf

\begin{remark}\label{zarbi} Lemma \ref{Fredholm} has the following rather curious consequence: if a typical $T\in(\bbx,{\sot})$ has closed range, then a typical $T\in\bbx$ has an infinite-dimensional kernel.
\end{remark}

\begin{proof}[\it Proof of Theorem \ref{Eisner-Matrai}] In this proof, the word ``typical'' refers to the topology \sot.
By the Wold Decomposition Theorem, an operator $T\in\mathcal B_1(\ell_2)$ is unitarily equivalent to $B_\infty$ if and only if $T$ is a co-isometry with infinite-dimensional kernel and $T^n\xrightarrow{\sot} 0$ as $n\to \infty$. So, by the Baire Category Theorem, it is enough to prove separately the following three facts: a typical $T\in\mathcal B_1(\ell_2)$ is a  co-isometry; a typical $T\in\mathcal B_1(\ell_2)$ is such that $T^n\xrightarrow{\sot} 0$ as $n\to\infty$; and a typical $T\in\mathcal B_1(\ell_2)$ is such that $\dim\ker(T)=\infty$. The first fact follows from part (3) of Proposition \ref{CPH} and Corollary \ref{CPCom}. The second fact is Lemma \ref{C0}. As to the third fact, note that a typical $T\in\mathcal B_1(\ell_2)$ is surjective since $T^*$ is an isometry, and hence a typical $T\in\mathcal B_1(\ell_2)$ has closed range. By Remark \ref{zarbi}, it follows that a typical $T\in\mathcal B_1(\ell_2)$ is such that $\dim\ker(T)=\infty$.
\epf

\begin{remark} As shown in \cite[Proposition 2.7 and Corollary 6.2]{EM}, the situation is quite the opposite if one considers the topologies \wot\ and \sote\ rather than \sot; namely, all orbits of unitary equivalence are \wot$\,$-$\,$meager and \sote -$\,$meager.
\end{remark}

\smallskip To conclude this section, we show that Theorem \ref{Eisner-Matrai} fails on $\ell_p$ for any  $p\neq 2$. 
\bpr\label{non-EM} Let $X=\ell_p$, $1\leq p\neq 2<\infty$. Denote by ${\rm Iso}(X)$ the group of all surjective  linear isometries of $X$. For any $A\in\bbx$, the set $\bigl\{ J^{-1} AJ;\; J\in{\rm Iso}(X)\bigr\}$ is meager in $\bigl(\bbx,\emph{\sot}\bigr)$.
\epr

\smallskip The proof of Proposition \ref{non-EM} relies on the following lemma.
\blm\label{nonsense} Let $X$ be a {\rm (}separable{\rm )} Banach space, and let $\mathcal A\subseteq \bbx$. Assume that there exists some vector $e\neq 0$ in $X$ such that the set $\{ Te;\; T\in\mathcal A\}$ is meager in $X$. Then $\mathcal A$ is meager in $\bigl(\bbx,\tau\bigr)$ for any topology $\tau$ on $\bbx$ finer than $\emph{\sot}$ and weaker than the operator norm topology.
\elm
\bpf We may assume that $\Vert e\Vert =1$. Let us denote by $B_X$ the closed unit ball of $X$ and by $\delta_e:\bbx\to B_X$ the evaluation map defined by $\delta_e(T):=Te$.
\begin{claim}\label{vraiment?} The map $\delta_e:\bbx\to B_X$ is
is continuous on $(\bbx,\tau)$ and, for any non-empty $\tau\,$-$\,$open set $\mathcal O\subseteq \bbx$, the set $\delta_e(\mathcal O)$ has non-empty interior in $B_X$.
\end{claim}
\begin{proof}[Proof of Claim \ref{vraiment?}]  
That $\delta_e$ is continuous on $(\bbx,\tau)$ is clear since $\tau$ is finer than \sot. Let $\mathcal O\subseteq \bbx$ be non-empty and $\tau\,$-$\,$open. Since $\tau$ is weaker than the operator norm topology, we can choose $T_0\in\mathcal O$ such that $\Vert T_0\Vert<1$, and then $\delta>0$ such that the open ball $B(T_0,\delta)$ is contained in $\mathcal O$. Let $x_0:= T_0e$. For any $x\in B(x_0,\delta)$, the operator $T:= T_0+e^*\otimes(x-x_0)$ is in $B(T_0,\delta)$, hence in $\mathcal O$, and satisfies $Te=x$. This shows that $\delta_e(\mathcal O)$ has non-empty interior.
\epf
It follows from Claim \ref{vraiment?} that for any meager set $M\subseteq B_X$, the set $\delta_e^{-1}(M)$ is meager in $(\bbx,\tau)$: indeed, if $F\subseteq B_X$ is a closed set with empty interior, then $\delta_e^{-1}(F)$ is closed in $(\bbx,\tau)$, and it must have empty interior because $\delta_e\bigl(\delta_e^{-1}(F)\bigr)\subseteq F$. Taking $M:=\{ Te;\; T\in\mathcal A\}=\delta_e(\mathcal A)$, we immediately get that $\mathcal A$ is meager in $(\bbx,\tau)$.
\epf

\begin{remark} Let $B_\infty$ be the backward shift with infinite multiplicity acting on $H=\ell_2(\Z_+, \ell_2)$. By Theorem \ref{Eisner-Matrai} and Lemma \ref{nonsense}, we see that for any $u\in H\setminus\{ 0\}$, the set $\Omega_u:=\{ J^{-1} B_\infty J u;\; J\in{\rm Iso}(H)\}$ is non-meager in $H$. In fact, one can check directly that $\Omega_u$ contains every vector $v\in H$ such that $\Vert v\Vert\leq \Vert u\Vert$ and $\sqrt{\Vert u\Vert^2-\Vert v\Vert^2}\, \Vert v\Vert \geq \vert \langle u,v\rangle\vert$, and hence that $\Omega_u$ has non-empty interior. Indeed, since $\dim\ker(B_\infty)\geq 2$, these conditions allow one to find $v', z\in\ker(B_\infty) $ such that $\Vert v'\Vert =\Vert v\Vert$, $\Vert z\Vert =\sqrt{\Vert u\Vert^2-\Vert v\Vert^2}$ and $\langle z,v'\rangle=\langle u,v\rangle$. Then $u':= z+B_\infty^* v'$ satisfies $\Vert u'\Vert =\Vert u\Vert$ and $B_\infty u'=v'$ because $B^*_\infty$ is an isometry,  and $\langle u',v'\rangle=\langle u,v\rangle$. Since $\Vert u'\Vert =\Vert u\Vert$, $\Vert v'\Vert=\Vert v\Vert$ and $\langle u',v'\rangle=\langle u,v\rangle$, one can find $J\in{\rm Iso}(H)$ such that $Ju=u'$ and $Jv=v'$, and we have $B_\infty Ju=Jv$.
\end{remark}

\begin{proof}[Proof of Proposition \ref{non-EM}] Let $A\in\mathcal B_1(\ell_p)$. By Lemma \ref{vraiment?}, it is enough to find some vector $e\neq 0$ in $\ell_p$ such that the set $\bigl\{ J^{-1}AJe;\; J\in{\rm Iso}(\ell_p)\bigr\}$ is meager in $\ell_p$. We take $e:=e_0$, the first vector of the canonical basis of $\ell_p$.

It is well known (see \mbox{e.g.} \cite[Proposition 2.f.14]{LT} that if $J\in{\rm Iso}(\ell_p)$, then there exist a permutation $\sigma:\Z_+\to\Z_+$ and scalars $\lambda_j\in\T$, $j\in\Z_+$, such that 
\[ \forall j\in\Z_+\;:\; Je_j=\lambda_j e_{\sigma(j)}.\]
We then have
\[ (J^*)^{-1}e_k^*=\lambda_{k}^{-1} e_{\sigma(k)}^*\qquad\hbox{for all $k\in\Z_+$},\]
and hence
\[ \langle e_0^*, J^{-1}AJe_0\rangle=\langle e_{\sigma(0)}^*, A e_{\sigma(0)}\rangle.\]

So, we see that the set $\bigl\{  \langle e_0^*, J^{-1}AJe_0\rangle;\; J\in{\rm Iso}(\ell_p)\bigr\}$ is countable, and hence meager in $\K$. Since the non-zero linear functional $e_0^*:\ell_p\to \K$ is continuous and open, it follows that $\{ J^{-1}AJe_0;\; J\in{\rm Iso}(\ell_p)\}$ is meager in $\ell_p$.
\epf

\section{Operators which do not attain their norm}\label{NAsection} This section is  motivated by Corollary \ref{KK&meager} and also by the theory of \emph{norm-attaining} operators. 
We recall that an operator $T\in\bbx$ on a (real or complex) separable Banach space $X$ is said to be {norm-attaining} if $\mathcal N(T)\neq\emptyset$, \mbox{\it i.e.} there exists $z\neq 0$ such that $\Vert Tz\Vert =\Vert z\Vert$; equivalently, if there exists $x\in B_X$ such that $\Vert Tx\Vert=\Vert T\Vert$.

\smallskip  Let $\mathbf M:=\bbx$. Assuming that $X$ has the Kadec-Klee property, Corollary \ref{KK&meager} says that if one wants to prove that $(\bbx,\wot)$ and $(\bbx,\sot)$ have the same comeager sets, it is enough to show that there is an \sot$\,$-$\,$dense set of $T\in\bbx$ having many norming vectors. So it may look natural to ask whether a \emph{typical} $T\in\bigl(\bbx,\sot\bigr)$ has many norming vectors. 

Note that if one considers the operator norm topology rather than the topology \sot, then it is well known that the answer is negative. Indeed, it follows from a classical result of Bourgain \cite{B} that if $X$ is a Banach space with the Radon-Nikod\'ym property, then a typical $T\in(\bbx,\Vert\,\cdot\,\Vert)$ is \emph{absolutely strongly exposing}, which means the following: there exists $x\in S_X$ such that, whenever $(x_n)\subseteq B_X$ is a maximising sequence for  $T$, \mbox{\it i.e.} $\Vert Tx_n\Vert\to \Vert T\Vert$, one can find a sequence $(\omega_n)\subseteq\T$ such that $\omega_n x_n\to x$. This implies in particular that $T$ is norm-attaining and attains its norm at a unique point of $S_X$ (up to rotation). We refer to \cite{JMR} and the references therein for much more on this topic.

\smallskip Somewhat surprisingly, our next result will show that for $X=\ell_p$, $1<p<\infty$, and \emph{with respect to the topology} \sote, the typical behaviour is in fact to have no norming vector at all. This, of course, does not exclude the possibility of finding an \sot$\,$-$\,$dense (or even \sote-$\,$dense) set of operators $T$ with many norming vectors. 

\smallskip Before stating the result, let us introduce some terminology. 

\smallskip The Banach space $X$ is said to be \emph{smooth} if its norm is Gateaux differentiable outside the origin; equivalently, if for  every $x\in X\setminus\{ 0\}$, there is only one linear functional $x^*\in X^*$ such that $\Vert x^*\Vert=1$ and $\langle x^*,x\rangle=\Vert x\Vert$.

\smallskip
By a \emph{duality mapping} for $X$ we mean any map $\mathbf J:X\to X^*$ with the following properties:
\be
\item[-] $\mathbf J(0)=0$ and $\mathbf J(x)\neq 0$ for all $x\neq 0$,
\item[-] 
$\langle \mathbf J(x),x\rangle=\Vert \mathbf J(x)\Vert \Vert x\Vert$ for all $x\in X$,
\item[-] $\Vert \mathbf J(x)\Vert$ depends only on $\Vert x\Vert$ and is a non-decreasing function of $\Vert x\Vert$.
\ee

\noindent
Note that if $X$ is smooth, then there is a ``canonical'' duality mapping $\mathbf J_0:X\to X^*$; namely, $\mathbf J_0(0)=0$ and $\mathbf J_0(x)$ is the Gateaux differential of the norm at $x$ if $x\neq 0$, \mbox{\it i.e.} the unique linear functional such that $\Vert \mathbf J_0(x)\Vert =1$ and $\langle \mathbf J_0(x),x\rangle=\Vert x\Vert$. In this case, any duality mapping has the form $\mathbf J(x)=c(\Vert x\Vert) \mathbf J_0(x)$ for some non-decreasing function $c:[0,\infty)\to[0,\infty)$.

\smallskip 
We will say that a projection operator $P\in\mathcal B(X)$ is \emph{$M$-$\,$like} if for every $T\in\mathcal B(X)$, it holds that \[ \Vert PTP+(I-P)T(I-P)\Vert=\max\bigl( \Vert PTP\Vert, \Vert (I-P)T(I-P)\Vert\bigr).\] 
(Note that this forces $\Vert P\Vert =1=\Vert I-P\Vert$.) For example, if $X=\ell_p$ or $c_0$ and $\Lambda\subseteq\Z_+$, the canonical projection $P_\Lambda$ onto $\overline{\rm span}\,\{ e_j;\; j\in \Lambda\}$ is $M$-$\,$like. More generally, any $L_p\,$-$\,$projection, $1\leq p\leq\infty$ is $M$-$\,$like. The terminology is justified by the following fact: denoting by $\mathcal D_P(X)$ the subspace of $\mathcal B(X)$ consisting of all operators which are block-diagonal with respect to the decomposition $X= {\rm ran}(P)\oplus \ker(P)$, the projection $P$ is $M$-$\,$like if and only if the canonical projection $\pi : A\oplus B\mapsto A$ is an  $M$-$\,$projection on $\mathcal D_P(X)$ in the usual sense, \mbox{\it i.e.} $\Vert T\Vert=\max\bigl( \Vert \pi (T)\Vert, \Vert (I-\pi)(T)\Vert\bigr)$ for every $T\in\mathcal D_P(X)$. 

\bth\label{notNA} Let $X$ be a reflexive {\rm (}separable{\rm )} Banach space. Assume that $X$ is smooth, that there exists a duality mapping $\mathbf J:X\to X^*$ which is $(w,w^*)\,$-$\,$continuous on bounded sets, and that there exists a sequence $(P_N)_{N\geq 0}$ of finite rank $M$-$\,$like projections on $X$  such that $P_N\xrightarrow{\emph{\sote}} I$. Then a typical $T\in(\bbx,\emph{\sote})$ does not attain its norm.
\eth

From this result, we deduce
\bco\label{quelquechose} For any $1<p<\infty$, a typical $T\in\bigl(\mathcal B_1(\ell_p),\emph{\sote}\bigr)$ does not attain its norm.
\eco
\bpf We only have to check that there exists a duality mapping $\mathbf J:\ell_p\to \ell_q$ which is $(w,w^*)\,$-$\,$continuous on bounded sets. This is quite well known. For $x=(x_j)_{j\geq 0}\in\ell_p$, define $\mathbf J(x)$ as follows:
\[ \mathbf J(x)_j:= \frac{\vert x_j\vert^p}{x_j},\quad\ j\in\Z_+,\]
with the natural convention $\frac{\vert 0\vert^p}{0}=0$. Then $\mathbf J(x)\in \ell_q$ with $\Vert \mathbf J(x)\Vert_q=\Vert x\Vert_p^{p/q}=\Vert x\Vert_p^{p-1}$ and $\langle \mathbf J(x), x\rangle=\Vert x\Vert_p^p$, so $\mathbf J$ is a duality mapping. Moreover, $\mathbf J$ is $(w,w^*)\,$-$\,$continuous on bounded sets because it is bounded on bounded sets and $w^*$-$\,$convergence is equivalent to coordinate-wise convergence on bounded subsets of $\ell_q$.
\epf

\begin{remark} On can replace the space $\ell_p$ in the statement of Corollary \ref{quelquechose} by any Banach space $X$ which is an $\ell_p\,$-$\,$direct sum of a sequence  $(E_j)_{j\geq 0}$ of smooth finite-dimensional spaces. The duality mapping $\mathbf J:X\to X^*$ with the required continuity property is defined as follows: if $x=(x_j)_{j\geq 0}\in X$, then $\mathbf J(x)=(\mathbf J_j(x_j))_{j\geq 0}$, where $\mathbf J_j(x_j)\in E_j^*$ is the unique linear functional satisfying $\langle \mathbf J_j(x_j),x_j\rangle=\Vert x_j\Vert^p$ and $\Vert \mathbf J_j(x_j)\Vert=\Vert x_j\Vert^{p-1}$.  
\end{remark}

\smallskip The proof of Theorem \ref{notNA} relies on the next two lemmas. In what follows, we denote by $\mathcal S(X)$ the unit sphere of $\mathcal B(X)$. Our first lemma is  well known.
\blm\label{S} If $X$ is any infinite-dimensional Banach space, then $\mathcal S(X)$ is \emph{\sot}$\,$-$\,G_\delta$ and \emph{\sote}-$\,$dense in $\bbx$. 
\elm
\bpf It is clear that $\mathcal S(X)$ is \sot$\,$-$\,G_\delta$ since if $T\in\bbx$, then \[ T\in\mathcal S(X)\iff\forall n\in\N\;\exists x\in B_X\;:\;\Vert Tx\Vert >1-\frac1n\cdot\]

\smallskip In order to prove that $\mathcal S(X)$ is {\sote}-$\,$dense in $\bbx$, it is enough to show the following: for any $A\in\bbx$  and for any $x_1,\dots ,x_N\in X$ and $x_1^*,\dots, x_N^*\in X^*$, one can find $T\in\mathcal S(X)$ such that $Tx_i=Ax_i$ and $T^*x^*_i=A^*x_i^*$ for $i=1,\dots ,N$. Choose $x^*\in X^*\setminus\{ 0\}$ such that $\langle x^*, x_i\rangle=0$ for $i=1,\dots ,N$ and $x\in X\setminus\{ 0\}$ such that $\langle x^*_i, x\rangle=0$ for $i=1,\dots ,N$. Since $\Vert A\Vert\leq 1$, one can find $\varepsilon \geq 0$ such that $T:=A+\varepsilon\, x^*\otimes x$ satisfies $\Vert T\Vert =1$; and $T$ has the required properties.
\epf

Our second lemma is essentially a rephrasing of \cite[Theorem 2.2]{Debm2}.

\blm\label{J} Assume that $X$ is smooth, and let $\mathbf J:X\to X^*$ be any duality mapping. If $T\in\mathcal S(X)$ and $x\in X\setminus\{ 0\}$, then
\[ x\in\mathcal N(T)\iff T^*\bigl(\mathbf J(Tx)\bigr)=\mathbf J(x).\]
\elm
\bpf Assume that $x\in \mathcal N(T)$, \mbox{\it i.e.} $\Vert Tx\Vert = \Vert x\Vert$. Then 
\[ \langle T^*\bigl(\mathbf J(Tx)\bigr), x\rangle=\langle \mathbf J(Tx), Tx\rangle=\Vert \mathbf J(Tx)\Vert \Vert Tx\Vert=\Vert \mathbf J(Tx)\Vert  \Vert x\Vert.\]
Since $\Vert T^*\bigl(\mathbf J(Tx)\bigr)\Vert \leq \Vert \mathbf J(Tx)\Vert$, it follows that in fact $\Vert T^*\bigl(\mathbf J(Tx)\bigr)\Vert = \Vert \mathbf J(Tx)\Vert$ and that $\langle T^*\bigl(\mathbf J(Tx)\bigr), x\rangle=\Vert T^*\big(\mathbf J(Tx)\bigr)\Vert \Vert x\Vert$. So $T^*\big(\mathbf J(Tx)\bigr)$ is a non-negative multiple of $\mathbf J_0(x)$, where $\mathbf J_0$ is the canonical duality mapping of $X$, and hence $T^*\bigl(\mathbf J(Tx)\bigr)= c \mathbf J(x)$ for some $c\geq 0$. Finally, we have $c=\Vert T^*\bigl(\mathbf J(Tx)\bigr)\Vert /\Vert \mathbf J(x)\Vert= \Vert \mathbf J(Tx)\Vert/\Vert \mathbf J(x)\Vert$, hence $c=1$ since $\Vert \mathbf J(u)\Vert$ depends only on $\Vert u\Vert$ and $\Vert Tx\Vert=\Vert x\Vert$.

\smallskip Conversely, assume that $T^*\bigl(\mathbf J(Tx)\bigr)=\mathbf J(x)$. Then
\[ \Vert \mathbf J(x)\Vert \Vert x\Vert =\langle \mathbf J(x),x\rangle= \langle \mathbf J(Tx), Tx\rangle =\Vert \mathbf J(Tx)\Vert \Vert Tx\Vert;\]
and since $\Vert Tx\Vert\leq \Vert x\Vert$ and $\Vert \mathbf J(u)\Vert$ is a non-decreasing function of $\Vert u\Vert$, it follows that we must have $\Vert Tx\Vert =\Vert x\Vert$.
\epf

\begin{proof}[Proof of Theorem \ref{notNA}] Let us denote by ${\rm NA}(X)$ the set of all $T\in\mathcal B(X)$ which attain their norm. By Lemma \ref{S}, it is enough to show that ${\rm NA}(X)\cap \mathcal S(X)$ is meager in $(\mathcal S(X), \sote)$. We are going to show that ${\rm NA}(X)\cap \mathcal S(X)$ is \sote-$\,F_\sigma$ in $\mathcal S(X)$, and that $\mathcal S(X)\setminus{\rm NA}(X)$ is \sote-$\,$dense in $\mathcal S(X)$.

\medskip 
To show that ${\rm NA}(X)\cap \mathcal S(X)$ is \sote-$\,F_\sigma$ in $\mathcal S(X)$, we start with the following observation.
\begin{claim}\label{closedset} For any bounded set $B\subseteq X$, the map $(T,x)\mapsto T^*\bigl(\mathbf J(Tx)\bigr)$ is continuous from $(\bbx,\sote)\times (B,w)$ into $(X^*, w^*)$. 
\end{claim}
\begin{proof}[Proof of Claim \ref{closedset}] Let $(T_n,x_n)$ be a sequence of elements of $\bbx\times B$ such that $T_n\xrightarrow{\sote} T\in\bbx$ and $x_n\xrightarrow{w} x\in B$. Then  $T_nx_n\xrightarrow{w}Tx$: indeed, if $x^*\in X^*$ then $\langle x^*, T_nx_n\rangle=\langle T_n^*x^*,x_n\rangle\to \langle T^*x^*, x\rangle=\langle x^*,Tx\rangle$ because $x_n\xrightarrow{w}x$ and $T^*_nx^*\xrightarrow{\Vert\,\cdot\,\Vert}T^*x^*$. Since $\mathbf J$ is $(w,w^*)\,$-$\,$continuous on bounded sets, it follows that $\mathbf J(T_nx_n)\xrightarrow{w^*}\mathbf J(Tx)$; and since $T_nz\xrightarrow{\Vert\,\cdot\,\Vert}Tz$ for all $z\in X$, this implies (as above) that $T^*_n\bigl( \mathbf J(T_nx_n)\bigr)\xrightarrow{w^*} T^*\bigl( \mathbf J(Tx)\bigr)$.
\end{proof}

\smallskip
Let $(B_q)_{q\in\N}$ be a family of closed balls in $X$ such that $\bigcup_{q\in\N} B_q=X\setminus\{ 0\}$. By Lemma \ref{J}, if $T\in\mathcal S(X)$ then 
\[ T\in {\rm NA}(X)\iff \exists q\in\N\;\exists x\in B_q\;:\; T^*\bigl(\mathbf J(Tx)\bigr)=\mathbf J(x).
\]
By Claim \ref{closedset} and since $\mathbf J$ is $(w,w^*)\,$-$\,$continuous on bounded sets, we see that for each $q\in\N$, the set $\{ (T, x)\in\mathcal S(X);\; T^*\bigl(\mathbf J(Tx)\bigr)=\mathbf J(x)\}$ is closed in $(\mathcal S(X),\sote)\times (B_q,w)$; and since $B_q$ is weakly compact, it follows that the set $\{ T\in\mathcal S(X);\; \exists x\in B_q\;:\; T^*\bigl(\mathbf J(Tx)\bigr)=\mathbf J(x)\}$ is closed in $(\mathcal S(X),\sote)$. Hence, ${\rm NA}(X)\cap\mathcal S(X)$ is $F_\sigma$ in $(\mathcal S(X),\sote)$.

\medskip Let us now show that $\mathcal S(X)\setminus{\rm NA}(X)$ is \sote-$\,$dense in $\mathcal S(X)$. 

First, note that for any $N\geq 0$, the space   $F_N:={\rm ran}(I-P_N)$ is reflexive, infinite-dimensional, and has the Approximation Property (even the Metric Approximation Property) since the sequence of finite rank operators $\bigl( (I-P_N) (P_k)_{| F_N}\bigr)_{k\geq 0}$ converges \sot\ to $I_{F_N}$. By a classical result of Holub \cite{Hol}, it follows that there is an operator $B_N\in\mathcal B(F_N)$ which does not attain its norm; and of course we may assume that $\Vert B_N\Vert=1$.

Now, let $A\in\mathcal S(X)$ be arbitrary. Let $(\varepsilon_N)_{N\geq 0}$ be any sequence of positive numbers tending to $0$. For each $N\geq 0$, considering $B_N$ as an operator from $F_N$ to $X$, define \[ T_N:=\left(1-\varepsilon_N\right) P_NAP_N+ (I-P_N)B_N(I-P_N)\in \mathcal B(X).\]

Since $P_N$ is $M$-$\,$like, we have $\Vert T_N\Vert=1$, \mbox{\it i.e.} $T_N\in\mathcal S(X)$, and $T_N\xrightarrow{\sote} A$ because $P_N\xrightarrow{\sote} I$ and the sequence $(B_N)$ is bounded. So it is enough to show that $T_N\not\in {\rm NA}(X)$. For this, we will use the following fact.
\begin{fact}\label{strict} Let $P\in\mathcal B(X)$ be a $M$-$\,$like projection. If $u, v\in X$ are such that $\Vert Pu\Vert<\Vert Pv\Vert$ and $\Vert (I-P)u\Vert< \Vert (I-P)v\Vert$, then $\Vert u\Vert <\Vert v\Vert$.
\end{fact}
\begin{proof}[Proof of Fact \ref{strict}] We may assume that $Pu\neq 0$ and $(I-P)u\neq 0$. Since $\Vert Pu\Vert <\Vert Pv\Vert$, one can find $x^*\in X^*$ such that $\Vert x^*\Vert<1$ and $\langle x^*, Pv\rangle =\Vert Pu\Vert$. Then the operator 
$A:= x^*\otimes \frac{Pu}{\Vert Pu\Vert}$ satisfies $\Vert A\Vert <1$ and $APv=Pu$. Similarly, one can find $B\in\mathcal B(X)$ such that $\Vert B\Vert <1$ and $B(I-P)v=(I-P)u$. Then $T:= PAP+(I-P)B(I-P)$ satisfies $\Vert T\Vert <1$ because $P$ is $M$-$\,$like, and $Tv=u$. Hence $\Vert u\Vert <\Vert v\Vert$.
\epf

Going back to $T_N$, let us show that $T_N\not\in{\rm NA}(X)$. Let $x\in X\setminus\{ 0\}$: we have to show that $\Vert Tx\Vert <\Vert x\Vert$. If $P_Nx=0$ then $x\in F_N$ and $T_Nx=B_Nx$, so $\Vert T_Nx\Vert <\Vert x\Vert$ because $B_N$ does not attain its norm. If $(I-P_N)x=0$, then $T_Nx=\left(1-\varepsilon_N\right) Ax$, so we have  $\Vert T_Nx\Vert \leq \left(1-\varepsilon_N\right)\Vert x\Vert<\Vert x\Vert$. Finally, assume that $P_Nx\neq 0$ and $(I-P_N)x\neq 0$. Then $\Vert P_NT_Nx\Vert =\Vert\left(1-\varepsilon_N\right)P_NAP_Nx\Vert <\Vert P_Nx\Vert$, and $\Vert (I-P_N)T_Nx\Vert =\Vert (I-P_N)B_N(I-P_N)x\Vert<\Vert (I-P_N)x\Vert$ because $B_N\not\in{\rm NA}(F_N)$.  Hence $\Vert T_Nx\Vert <\Vert x\Vert$ by Fact \ref{strict}.
\epf


\begin{remark} Let $X=\ell_p$, $1<p<\infty$. The proof of Theorem \ref{notNA} has shown that ${\rm NA}(X)\cap \mathcal S(X)$ is $F_\sigma$ in $(\mathcal S(X),\sote)$, hence in $(\mathcal S(X),\Vert\,\cdot\,\Vert)$. On the other hand, by Bourgain's result \cite{B} mentioned above, ${\rm NA}(X)\cap \mathcal S(X)$ is comeager in $(\mathcal S(X), \Vert\,\cdot\,\Vert)$. So, we obtain that $\mathcal S(X)\setminus {\rm NA}(X)$ is in fact \emph{nowhere dense} in $(\mathcal S(X),\Vert\,\cdot\,\Vert)$. In other words, ${\rm NA}(\ell_p)\cap \mathcal S(\ell_p)$ contains a norm-dense open subset of $\mathcal S(\ell_p)$; from which it follows that ${\rm NA}(\ell_p)$ contains a norm-dense open subset of $\mathcal B(\ell_p)$.
\end{remark}

\smallskip It follows from Theorems \ref{notNA} and \ref{samemeager} that if $2<p<\infty$, a typical $T\in\bigl(\mathcal B_1(\ell_p), \sot\bigr)$ does not attain its norm. This leads to the following question:
\begin{question} Let $1<p<2$. Is it true that a typical $T\in\bigl(\mathcal B_1(\ell_p), \sot\bigr)$ does not attain its norm?
\end{question}

\section{Additional results}\label{FredhSection}

\subsection{More on similarity of topologies} Let $\mathbf M$ be an abstract set. We have observed that two topologies on $\mathbf M$ sharing the same comeager sets may not be similar.  In the next proposition, we show that if some extra assumptions on the topologies are added, then the two properties become equivalent. 
 
 \bpr\label{similarcomeagerbis} Let $\tau$ and $\tau'$ be two topologies on $\mathbf M$. Assume that $\tau\subseteq \tau'$, that $\tau'$ is Baire and second-countable, and that the identity map $\mathbf i_{\tau, \tau'}:(\mathbf M,\tau)\to (\mathbf M,\tau')$ is Borel $1$, \textit{i.e.} any $\tau'$-$\,$open subset of $\mathbf M$ is $\tau\,$-$\,F_\sigma$. Then $\tau$ and $\tau'$ are similar if and only if they have the same comeager sets.
 \epr
 \bpf By Lemma \ref{similarcomeager}, we only need to prove the ``if'' part. Assume that $\tau$ and $\tau'$ have the same comeager sets. Since $\mathbf i_{\tau,\tau'}$ is Borel $1$ and $\tau'$ is second-countable, we know that $\mathcal C(\tau,\tau')$ is $\tau\,$-$\,$comeager. So $\mathcal C(\tau,\tau')$ is $\tau'\,$-$\,$comeager by our assumption, and hence $\tau'\,$-$\,$dense in $\mathbf M$ since 
 $\tau'$ is Baire. By Lemma \ref{tropsimple}, it follows that $\tau$ and $\tau'$ are similar. (Note that it is not really necessary to assume that $\tau\subseteq \tau'$: it is enough to know that $\mathcal C(\tau',\tau)$ is $\tau\,$-$\,$dense.)
 \epf
 
  One gets the same conclusion with the following ``symmetric'' assumptions: $\tau$ and $\tau'$ are both Baire and second-countable, and the two identity maps are Borel $1$.

\smallskip We now prove a result that allows to give a few other characterisations of similarity.
\bpr\label{symmetric} Let $\tau$ and $\tau'$ be similar topologies on $\mathbf M$, and assume that $\tau$ and $\tau'$ are second-countable. Then $\mathcal C(\tau,\tau')\cap \mathcal C(\tau',\tau)$ is comeager {\rm (}for both $\tau$ and $\tau'${\rm )}.
\epr
\bpf It is enough to show that $\mathcal A:=\mathcal C(\tau,\tau')\cap \mathcal C(\tau',\tau)$ is $\tau\,$-$\,$comeager; and to this end we use the \emph{Banach-Mazur game} $\mathbf G(\mathcal A)$ associated with $\mathcal A$  (see \mbox{e.g.} \cite[p.\, 51]{Ke}), in the topological space $(\mathbf M,\tau)$. Recall that $\mathbf G(\mathcal A)$ is an infinite game with two players, denoted by I and II. The two players play alternatively non-empty $\tau\,$-$\,$open sets $ U_0\supseteq  U_1\supseteq  U_2\supseteq \cdots$, and player II wins the run if $\bigcap_{n\geq 0}  U_n\subseteq \mathcal  A$. To prove that $\mathcal A$ is $\tau\,$-$\,$comeager, it is enough to show that player II has a winning strategy in $\mathbf G(\mathcal A)$.

Let $(O_k)_{k\geq0}$ and $(O'_k)_{k\geq 0}$ be countable bases for $(\mathbf M,\tau)$ and $(\mathbf M,\tau')$. Assume that player I has just played a $\tau\,$-$\,$open set $U_{2k}\neq \emptyset$. Then II answers as follows. First, II chooses a $\tau'\,$-$\,$open set $W'_k\neq\emptyset$ such that $W'_{k}\subseteq U_{2k}$ and, moreover, $W'_{k}\subseteq U_{2k}\cap O_k$ if $U_{2k}\cap O_k\neq \emptyset$. This is possible since $\tau$ and $\tau'$ are similar. Then II plays a $\tau\,$-$\,$open set $U_{2k+1}\neq \emptyset$ such that $U_{2k+1}\subseteq W'_k$ and, moreover, $U_{2k+1}\subseteq W'_k\cap O'_k$ if $W'_k\cap O'_k\neq\emptyset$. Again, this is possible since $\tau$ and $\tau'$ are similar. Let us check that this strategy is winning for II.

Let $(U_n)_{n\geq 0}$ be a run of the game $\mathbf G(\mathcal A)$ where II has followed the above strategy, and let $x\in \bigcap_{n\geq 0} U_n$. We have to show that $x\in \mathcal A=\mathcal C(\tau,\tau')\cap\mathcal C(\tau',\tau)$ \mbox{\it i.e.} that any $\tau\,$-$\,$neighbourhood of $x$ is a $\tau'\,$-$\,$neighbourhood and vice-versa. Let $V$ be a $\tau\,$-$\,$neighbourhood of $x$. Choose $k\geq 0$ such that $x\in O_k\subseteq V$.  Then $x\in O_k\cap U_{2k}$, so that in particular $O_k\cap U_{2k}\neq\emptyset$. With the above notation, it follows that $W'_k\subseteq O_k\cap U_{2k}$ and hence $W'_k\subseteq V$. Since $x\in U_{2k+1}\subseteq W'_k$, this shows that $V$ is a 
$\tau'\,$-$\,$neighbourhood of $x$. Now, let $V'$ be a $\tau'\,$-$\,$neighbourhood of $x$. Choose $k\geq 0$ such that $x\in O'_k\subseteq V'$. Then $x\in U_{2k+1}\cap O'_k\subseteq W'_k\cap O'_k$, so $W'_k\cap O'_k\neq\emptyset$. Hence $U_{2k+1}\subseteq W'_k\cap O'_k\subseteq V'$; and since $x\in U_{2k+1}$, it follows that $V'$ is a $\tau\,$-$\,$neighbourhood of $x$.  
\epf

\bco Assume that $\tau$ and $\tau'$ are Baire and second-countable. Then the following assertions are equivalent. 
\be
\item[{\rm (i)}] $\tau$ and $\tau'$ are similar.
\item[\rm (ii)] $\mathcal C(\tau,\tau')\cap \mathcal C(\tau,\tau')$ is comeager for both $\tau$ and $\tau'$.
\item[\rm (iii)] $\mathcal C(\tau,\tau')$ and $\mathcal C(\tau,\tau')$ are both $\tau\,$-$\,$dense and $\tau'\,$-$\,$dense.
\item[\rm (iv)] $\mathcal C(\tau,\tau')$ is $\tau'\,$-$\,$dense and $\mathcal C(\tau,\tau')$ is $\tau\,$-$\,$dense.
\ee
\eco
\bpf This follows from Proposition \ref{symmetric} and Lemma \ref{tropsimple}.
\epf

In order to state our second corollary, we need to recall a perhaps not so well known notion. A subset $A$ of a topological space $E$ is said to be \emph{semi-open} if $A$ is contained in the closure of its interior; equivalently, if there exists an open set $O$ such that $O\subseteq A\subseteq\overline{O\,}$. This notion was introduced in \cite{Levine}, and independently in \cite{Norway} (where semi-open sets are called \emph{$\beta\,$-$\,$sets}). The somewhat dual property of being contained in the interior of its closure defines the class of \emph{locally dense} sets introduced in \cite{CM}. There is actually a rather abundant literature concerning various classes of ``generalised open sets''; see \mbox{e.g.} \cite{Csa}.
 \bco\label{equivalence} Assume that $\tau\subseteq\tau'$, that $\tau$ is second-countable, and that $\tau'$ is Baire and second-countable. 
 Then the following assertions are equivalent. 
 \be
 \item[\rm (a)] $\tau$ and $\tau'$ are similar;
 \item[\rm (b)] $\mathcal C(\tau,\tau')$ is $\tau'$-$\,$dense in $\mathbf M$;
  \item[\rm (c)] every $\tau'\,$-$\,$semi-open set is $\tau\,$-$\,$semi-open;
 \item[\rm (d)] every $\tau'\,$-$\,$open set is $\tau\,$-$\,$semi-open.
 \ee
 \eco
 \bpf 
 By Proposition \ref{symmetric} and since $\tau'$ is Baire, we know that (a)$\implies$(b).

 (b)$\implies$(c) Assume that $\mathcal C(\tau,\tau')$ is $\tau'$-$\,$dense in $\mathbf M$. Let $A\subseteq \mathbf M$ be $\tau'\,$-$\,$semi-open, and choose a $\tau'\,$-$\,$open set $O'$ such that $O'\subseteq A\subseteq \overline{O'}^{\,\tau'}$. Denote by $O$ the $\tau\,$-$\,$interior of $O'$, so that $O\subseteq O'\subseteq A$. Since $\mathcal C(\tau,\tau')$ is $\tau'$-$\,$dense in $\mathbf M$ and $O'$ is $\tau'\,$-$\,$open, $\mathcal C(\tau,\tau')\cap O'$ is $\tau'\,$-$\,$ dense in $O'$; and since $\tau\subseteq \tau'$, it follows that $ \overline{O'}^{\,\tau'}\subseteq \overline{\mathcal C(\tau,\tau')\cap O'}^{\,\tau}$ Moreover, $\mathcal C(\tau,\tau')\cap O'$ is contained in $O$ by the definition of $\mathcal C(\tau,\tau')$. So we see that $O\subseteq A\subseteq \overline{O}^{\,\tau}$, and hence $A$ is $\tau\,$-$\,$semi-open.
 
 (c)$\implies$(d) is obvious.
 
 (d)$\implies$(a) Assume (d). Since $\tau\subseteq \tau'$, it is enough to show that any set $A\subseteq \mathbf M$ with non-empty $\tau'\,$-$\,$interior has non-empty $\tau\,$-$\,$interior; which is obvious by (d).
 \epf
 
 \begin{remark} The implications (b)$\implies$(c)$\implies$(d)$\implies$(a) do not require $\tau$ and $\tau'$ to be second-countable or Baire. Moreover, without assuming that $\tau\subseteq\tau'$, the proof of (b)$\implies$(c) shows that if (b) holds true then, for any $\tau'\,$-$\,$semi-open set $A\subseteq \mathbf M$, one can find a $\tau\,$-$\,$open set $O$ such that $O\subseteq A\subseteq \overline{O}^{\,\tau'}$; and the proof of (d)$\implies$(a) shows that if (d) holds true then any $\tau\,$-$\,$dense set is $\tau'\,$-$\,$dense.\end{remark}
 
 \begin{remark} One cannot replace  condition (c) in Corollary \ref{equivalence} above by the condition ``$\tau$ and $\tau'$ have the same semi-open sets''. In fact, having the same semi-open sets should be considered as a much stronger property than similarity. For example, it follows from \cite[Propositions 1 and 8]{Norway}  that if $\tau$ and $\tau'$ are two \emph{regular} topologies with the same semi-open sets, then $\tau=\tau'$. (Recall that a topology is regular if every point of the space has a neighbourhood basis consisting of closed sets.)
 \end{remark}

\subsection{More on points of $\wot\,$-$\,\sot$ continuity} In this sub-section, we add some examples of situations where one can conclude that an operator $T\in\bbx$ belongs to $\mathcal C(\wot,\sot)$. 

\smallskip Our first example concerns isometries. Let us say that a projection operator $P\in\mathcal B(X)$ is \emph{uniformly strictly contractive} if $\Vert P\Vert=1$ and, whenever $(x_n)$ is a sequence in $B_X$ such that $\Vert Px_n\Vert\to 1$, it follows that $\Vert (I-P)x_n\Vert\to 0$. For example, any $L_p\,$-$\,$projection, $1\leq p<\infty$ is uniformly strictly contractive. (A projection $P\in\mathcal B(X)$ is an $L_p\,$-$\,$projection if $X={\rm ran}(P)\oplus_{\ell_p} \ker(P)$; see \cite{B&al}.)

\begin{example} Let $X$ be reflexive. If $T\in\bbx$ admits a left inverse $L\in\bbx$ such that the projection $P:=TL$ is uniformly strictly contractive, then $T\in \mathcal C({\wot},{\sot})$. In particular:
\be
\item[\rm (i)] any surjective isometry of $X$ belongs to $\mathcal C({\wot},{\sote})$;
\item[\rm (ii)] if $E$ is any reflexive space and $1<p<\infty$, the forward shift $S$ on $X:=\ell_p(\Z_+,E)$ belongs to $\mathcal C({\wot},{\sot})$.
\ee
\end{example}
\bpf  Note that the assumption made on $T$ implies that $T$ is an isometry since $\Vert T\Vert\leq 1$ and $\Vert L\Vert\leq 1$.

Since $X$ is reflexive and separable, the set $\mathcal C(w,\Vert\,\cdot\,\Vert)$ is weakly dense in $B_X$; in particular, we have $\overline{\rm span}\; \mathcal C(w,\Vert\,\cdot\,\Vert)=X$. Hence, by Proposition \ref{abstractnonsense}, it is enough to show that $T\bigl( \mathcal C(w,\Vert\,\cdot\,\Vert)\bigr)\subseteq \mathcal C(w,\Vert\,\cdot\,\Vert)$. 

Let $z\in \mathcal C(w,\Vert\,\cdot\,\Vert)$, and let $(x_n)$ be a sequence in $B_X$ such that $x_n\xrightarrow{w}Tz$. Since $LTz=z$ and $z\in\mathcal C(w,\Vert\,\cdot\,\Vert)$, we see that  $Lx_n\xrightarrow{\Vert \,\cdot\,\Vert} z$. Hence $\Vert TL x_n\Vert\to \Vert Tz\Vert=1$. Since $P=TL$ is uniformly strictly contractive, it follows that $\Vert (I-TL)x_n\Vert\to 0$; and since $TLx_n\xrightarrow{\Vert\,\cdot\,\Vert} Tz$, we conclude that $x_n\xrightarrow{\Vert\,\cdot\,\Vert} Tz$, which shows that $Tz\in\mathcal C(w,\Vert\,\cdot\,\Vert)$. 

\smallskip Since the identity operator is a uniformly strictly contractive projection, we see that any surjective isometry belongs to $\mathcal C({\wot},{\sot})$, and (i) follows by duality. As to (ii), take $L:=B$, the backward shift on $X=\ell_p(\Z_+,E)$. 
\epf

\smallskip Our second example follows from Proposition \ref{KK3}. Recall that a point $x\in S_X$ is said to be a \emph{strongly exposed point of $B_X$} if there exists a linear functional $x^*\in S_{X^*}$ such that $\langle x^*,x\rangle=1$ and $x^*_{| B_X}$ has a strong maximum at $x$, \mbox{\it i.e.} every maximising sequence $(x_n)\subseteq B_X$ for $x^*$ converges in norm to $x$.

\begin{example}\label{extreme} Let $X$ be reflexive with the Kadec-Klee property. If $T\in\bbx$ is such that $\Vert T\Vert=1$ and the weak closure of $\mathcal N(T)$ contains the strongly exposed points of $B_X$, then $T\in\mathcal C({\wot},{\sot})$. 
\end{example}
\bpf Since $X$ is reflexive, $B_X$ is the closed convex hull of its strongly exposed points, by \cite[Theorem 4]{Lind}. So the result is clear by Proposition \ref{KK3}.
\epf

\begin{remark*} It is quite possible that $\mathcal N(T)$ contains all extreme points of $B_X$ and yet $T$ is not an isometry. Here is a simple example. Let $1<p<\infty$ and let $X:=X_1\oplus_{\ell_p} X_2$, where $X_1:= (\K^2,\Vert \,\cdot\,\Vert_1)$ and $X_2$ is any reflexive Banach space with the Kadec-Klee property. Then $X$ is reflexive with the Kadec-Klee property, and the extreme points of $B_X$ are the points $x=x_1\oplus x_2$, where $\Vert x_1\Vert^p+\Vert x_2\Vert^p=1$ and $x_j\in {\rm Ext}\bigl(\Vert x_j\Vert B_{X_j}\bigr)$. Let $T:=A\oplus I_{X_2}\in\bbx$, where $A\in \mathcal B_1(X_1)$ is any non-isometric operator such that $\Vert Ae\Vert=\Vert Ae'\Vert=1$, where $(e,e')$ is the canonical basis of $X_1$. Then $T$ is not an isometry but $\mathcal N(T)\supseteq {\rm Ext}(B_X)$.
\end{remark*}

We now prove a variant of Proposition \ref{KK3} involving only ``{approximately norming}'' vectors. Recall that the Banach space $X$ is said to have the \emph{Uniform Kadec-Klee property} if the following holds true: for any $\varepsilon>0$, there exists $\delta>0$ such that, whenever $(x_n)$ is an $\varepsilon\,$-$\,$separated sequence in $B_X$ such that $x_n\xrightarrow{w} x\in X$, it follows that $\Vert x\Vert < 1-\delta$. (This property was introduced by Huff \cite{Hu}.) For example, any uniformly convex space is UKK. Also, if $T\in\mathcal B(X)$, let us say that a vector $z\in X\setminus\{ 0\}$ is \emph{$(1-\delta)\,$-$\,$norming} for $T$ is $\Vert Tz\Vert > (1-\delta) \Vert T\Vert \Vert z\Vert$. We denote by $\mathcal N_\delta(T)$ the set of all $(1-\delta)\,$-$\,$norming vectors for $T$.

\bpr\label{useless} Assume that $X$ is reflexive with the Uniform Kadec-Klee property, and let $T\in \bbx$. Assume that $\Vert T\Vert =1$ and that there exists some constant $c>0$ such that, for every $\delta >0$, the set $\overline{\rm conv}\,\bigl( \mathcal N_\delta(T)\cap S_X\bigr)$ contains the ball $cB_X$. Then, $T\in\mathcal C(\emph{\wot},\emph{\sot})$.
\epr
\bpf Let $(T_n)$ be a sequence in $\bbx$ such that $T_n\xrightarrow{\wot} T$, and let $x\in X$. We have to show that $T_nx\xrightarrow{\Vert\,\cdot\,\Vert} Tx$. So we fix $\alpha >0$, and we intend to show that $\Vert T_nx-Tx\Vert \leq\alpha$ for all large enough $n$. Moreover, we may assume that $\Vert x\Vert\leq 1$.

\begin{claim}\label{yan} Let $\varepsilon >0$, and let $\delta=\delta(\varepsilon)$ be given by the UKK property. For any $z\in\mathcal N_\delta(T)\cap S_X$, we have $\overline{\,\lim}\, \Vert T_n z-Tz\Vert \leq\varepsilon$.  
\end{claim}
\begin{proof}[Proof of Claim \ref{yan}]  Since $z\in\mathcal N_\delta(T)\cap S_X$, we have $\Vert Tz\Vert > 1-\delta$; and since $T_nz\xrightarrow{w} Tz$, it follows that the sequence $(T_nz)$ has no 
$\varepsilon\,$-$\,$separated subsequence. From that, it is not hard to deduce that one can find an increasing sequence of integers $(n_k)_{k\geq 0}$ such that $\Vert T_{n_k}z-T_{n_{k'}}z\Vert <\varepsilon$ for every $k,k'\geq 0$. Fixing $k$ and letting $k'\to\infty$, we get $\Vert T_{n_k} z-Tz\Vert\leq \varepsilon$ for all $k\geq 0$. So we have $\underline{\,\lim}\, \Vert T_n-Tz\Vert \leq \varepsilon$; and since this can be applied to any subsequence of $(T_n)$, the result follows. 
\epf
Now, let us choose $\varepsilon :=\alpha c/3$, and let $\delta$ be given by the UKK property.  By assumption on $T$, one can find  $u\in X$ which is a convex combination of vectors $z_1,\dots ,z_N\in \mathcal N_\delta(T)\cap S_X$ such that $\Vert u- cx\Vert<\alpha c/3$. By Claim \ref{yan}, we have $\Vert T_nz_i- Tz_i\Vert \leq \varepsilon=\alpha c/3$ for $i=1,\dots ,N$ if $n$ is large enough; and then $\Vert T_nu-Tu\Vert \leq \alpha c/3$ by convexity. By the triangle inequality and since $\Vert T\Vert, \Vert T_n\Vert\leq 1$, it follows that $\Vert T_n(cx)-T(cx)\Vert \leq \alpha c$ for all large enough $n$; so $\Vert T_nx -Tx\Vert \leq \alpha$, as required.
\epf

\begin{remark*} By the Hahn-Banach theorem and since $\mathcal N_\delta(T)$ is balanced for any $\delta>0$, the assumption in Proposition \ref{useless} can be written as follows: 
\[ \inf_{x^*\in S_{X^*}} \inf_{\delta >0}  \sup_{ z\in\mathcal N_\delta(T)\cap S_X}  \vert \langle x^*, z\rangle\vert >0.\]
\end{remark*}

\smallskip
We conclude this sub-section with a rather natural question. As observed in \cite[Proposition 2.6]{Kan}, if the Banach space $X$ is strictly convex then any operator $T\in\bbx$ such that $\Vert T\Vert=1$ and $\overline{{\rm span}}\bigl( \mathcal N(T)\bigr)=X$ (in particular, any isometry) has to be an \emph{extreme point} of $\bbx$. In fact, without any assumption on $X$, any $T\in\bbx$ such that $\Vert T\Vert=1$,  $\mathcal N(T)\cap S_X\subseteq {\rm Ext}(B_X)$ and  $\overline{{\rm span}}\bigl( \mathcal N(T)\bigr)=X$ is an extreme point of $\bbx$. By Proposition \ref{CPH}, it follows that for $X=\ell_2$, any $T\in\mathcal C(\wot,\sot)$ is an extreme point of $\bbx$. Also, it is easy to see that $\mathcal C(\wot,\sot)$ is an extreme subset of $\bbx$, for any $X$: if $A,B\in\bbx$ and $\frac{A+B}2\in\mathcal C(\wot,\sot)$, then $A\in \mathcal C(\wot,\sot)$ and $B\in \mathcal C(\wot,\sot)$. This motivates the following question.\begin{question} For which reflexive spaces $X$ is it true that every $T\in \mathcal C(\wot,\sot)$ is an extreme point of $\bbx$?
\end{question}

\subsection{Typical properties related to Fredholm theory} In this final sub-section, we gather a few remarks and some questions related to Lemma \ref{Fredholm}. Since we consider spectral properties, we assume that $\K=\C$. We begin by observing that the \sote\ version of {Lemma \ref{Fredholm} holds true if $X$ is reflexive and has the MAP}. In other words:
\begin{fact} If $X$ is reflexive and has the MAP then a typical $T\in(\bbx,\sote)$ is not upper semi-Fredholm.
\end{fact}
\bpf  By the proof of Lemma \ref{Fredholm}, it is in fact enough to show that the finite rank contraction operators are \sote-$\,$dense in $\bbx$; and this will follow if we can show that there exists a sequence of finite rank operators $(R_n)\subseteq \bbx$ such that $R_n\xrightarrow{\sote} I_X$. 

\smallskip
Since $X$ has the MAP, there is a sequence of finite rank operators $(A_k)\subseteq\bbx$ such that $A_k\xrightarrow{\sot} I_X$. Then $A_k\xrightarrow{\wot} I_X$, hence $A_k^*\xrightarrow{\wot}I_{X^*}$ since $X$ is reflexive. By Mazur's Theorem, it follows that $I_{X^*}$ belongs to the \sot$\,$-$\,$closure of the set $\bigl\{ A_k^*;\; k\geq n\bigr\}$, for any $n\in\N$. So, by metrizability of $\bigl( \mathcal B_1(X^*),\sot\bigr)$, one can find a sequence $(S_n)$ of convex combinations of the $A_k^*$ such that $S_n\xrightarrow{\sot} I_{X^*}$ and $S_n\in {\rm conv}\,\bigl\{ A_k^*;\; k\geq n\bigr\}$ for all $n\in\N$. The sequence of finite-rank operators $(R_n)$ defined by $R_n:=S_n^*$ is then such that $R_n\xrightarrow{\sot} I_X$ (because $A_k\xrightarrow{\sot} I_X$) and $R_n^*\xrightarrow{\sot} I_{X^*}$.
\epf

\smallskip The next proposition shows in particular that Lemma \ref{Fredholm} can be improved when $X$ is an $\ell_p\,$-$\,$direct sum or a $c_0\,$-$\,$direct sum of finite-dimensional spaces..
\bpr\label{sigmae} Assume that there exists a sequence $(P_N)_{N\geq 0}$ of finite rank $M$-$\,$like projections on $X$ such that $P_N\xrightarrow{\emph{\sot}}I$. Then a typical $T\in(\bbx,\emph{\sot})$ is such that $T-\lambda I$ is not upper semi-Fredholm for any $\lambda\in\overline{\,\D}$. In particular, the essential spectrum of a typical $T\in\bbx$ is equal to $\overline{\,\D}$.
\epr
\bpf We use an improved version of Fact \ref{EMn}.
\begin{fact}\label{EMnbis} Let $\lambda\in\overline{\,\D}$. A typical $T\in\bbx$ has the following property: for any $\varepsilon>0$ and any $n\in\N$, there exists a subspace $E\subseteq X$ with $\dim(E)> n$ such that $\Vert (T-\lambda I)_{| E}\Vert<\varepsilon$. 
\end{fact}
\begin{proof}[\it Proof of Fact \ref{EMnbis}] Let us denote by $\mathcal G$ the set of all $T\in\bbx$ with the above property. The set $\mathcal G$ is \sot$\,$-$\,\gd$; so we just have to check that $\mathcal G$ is \sot$\,$-$\,$dense in $\bbx$. Let $A\in\bbx$ be arbitrary, and for any $N\geq 0$, set $T_N:= P_NAP_N+\lambda (I-P_N)$. Since the projection $P_N$ is $M$-$\,$like, we see that $T_N\in\bbx$; and $T_N\xrightarrow{\sot} A$. Moreover, we have $T_N\equiv \lambda I$ on $E:=\ker(P_N)$, so $T\in\mathcal G$.
\epf

It follows from Fact \ref{EMnbis} that for any fixed $\lambda\in\overline{\,\D}$, a typical $T\in\bbx$ is such that $T-\lambda I$ is not upper semi-Fredholm. Now, let $\Lambda$ be a countable dense subset of $\overline{\,\D}$. By the Baire Category Theorem, a typical $T\in\bbx$ is such that $T-\lambda I$ is not upper semi-Fredholm for any $\lambda\in\Lambda$. Since the set of all $\lambda\in\C$ such that $T-\lambda I$ is not upper semi-Fredholm is closed in $\C$, the result follows.
\epf

\begin{remark} The \sote\ version of Proposition \ref{sigmae} holds true (with exactly the same proof) if one assumes that $P_N\xrightarrow{\sote} I$. This applies in particular to  $X=c_0$ or $\ell_p$, $1<p<\infty$.
\end{remark}

\begin{remark} Let $X=\ell_1$. It is proved in \cite{GMM2} that a typical $T\in(\bbx, \sot)$ has the following two properties: $T^*$ is an isometry, and every $\lambda\in\D$ is an eigenvalue of $T$ of infinite multiplicity. The proof of the  first property is rather simple, but the proof of the second property given in \cite{GMM2} is a bit technical. By Proposition \ref{sigmae}, we now see that the typicality of the second property actually \emph{follows} from the typicality of the first one. Indeed, if $T^*$ is an isometry, then $T-\lambda I$ is surjective for every $\lambda\in\D$ since $(T-\lambda I)^*$ is an embedding, and hence $\dim \ker(T-\lambda I)=\infty$ if $T-\lambda I$ is not Fredholm.
\end{remark}

\smallskip Assume that the Banach space $X$ has the MAP. As already pointed out, it follows from Lemma \ref{Fredholm} that if a typical $T\in\bbx$ has closed range, then a typical $T\in\bbx$ has an infinite-dimensional kernel. As the next remark shows, saying that a typical $T\in\bbx$ has closed range is in fact the same as saying that  a typical $T\in\bbx$ is surjective.
\begin{remark} If $X$ has the MAP, then a typical $T\in\bigl(\bbx, \sot\bigr)$ has dense range. Consequently, a typical $T\in\bbx$ has closed range if and only if a typical $T\in\bbx$ is surjective.
\end{remark}
\bpf  For any $z\in X$, set $ \mathcal G_z:=\bigl\{ T\in\bbx;\; z\in \overline{\,{\rm ran}(T)}\bigr\}$. Let also $D$ be a countable dense subset of $X$. Then an operator $T\in\bbx$ has dense range if and only if $T\in\bigcap_{z\in D} \mathcal G_z$. Moreover, each set $\mathcal G_z$ is easily seen to be \sot$\,$-$\,\gd$ in $\bbx$. So it is enough to show that $\mathcal G_z$ is \sot$\,$-$\,$dense in $\bbx$, for any $z\in X$.

Let $\mathcal U$ be a non-empty open set in $(\bbx, \sot)$. Since $X$ has the MAP, one can find a finite rank operator $A$ such that $A\in\mathcal U$ and $\Vert A\Vert <1$. Since $\dim(X)=\infty$, we can choose $x\in \ker(A)$ with $x\neq 0$, and then $x^*\in X^*$ such that $\pss{x^*}{x}=1$. Let $\varepsilon >0$, and set $T:= A+\varepsilon\, x^*\otimes z$. If $\varepsilon$ is small enough, then $\Vert T\Vert\leq 1$ and $T\in\mathcal U$. Moreover, since $Tx=\varepsilon z$, we have $z\in{\rm ran}(T)$ and hence $T_N\in\mathcal G_z$.
\epf

We finish the paper with the following two questions:

\begin{question}\label{qsurject} {Is a typical $T\in(\bbx,\sot)$ surjective when $X=\ell_p$, $1<p<2$?}
\end{question}

We have stated Question \ref{qsurject} for $1<p<2$ only because the answer is already known for all other values of $p$. Indeed, the answer is ``Yes'' for $X=\ell_2$ by \cite{EM}, and also ``Yes'' for $X=\ell_1$ by \cite[Theorem 4.1]{GMM2}. On the other hand, since a typical $T\in (\bbx,\sot)$ is not invertible (by \cite[Proposition 2.1]{GMM2}, see also Proposition \ref{sigmae}), the answer is ``No'' if $X=\ell_p$, $p>2$ or $c_0$, because a typical $T\in(\bbx,\sot)$ is one-to-one by \cite[Theorem 2.3]{GMM2}. The answer is also ``No'' for the topology \sote\ and any $1<p<\infty$, for the same reason. It follows from Theorem \ref{toutcapourca} that the answer is ``No'' as well for the topology \sot\ if $\K=\R$ and $p=3/2$.

\smallskip
\begin{question} Assume that $X$ has the MAP. Is it true that a typical $T\in\bigl(\bbx,\sot\bigr)$ is such that $T-\lambda I$ has dense range  for every $\lambda\in\C$?
\end{question}

%
%
%
%
%
%
%
%
%
%
%
%
%
%
\bigskip\noindent
\[ \hbox{{\bf Competing interests}: The authors declare none.}\]

\end{document}